\documentclass[12pt]{amsart}
\usepackage{graphicx}
\usepackage{amssymb}
\usepackage{enumitem}
\usepackage{epstopdf}
\usepackage{amsmath, amssymb, amsfonts, amsthm,amscd}

\newtheorem{lemma}{Lemma}
\newtheorem{algorithm}{Algorithm}
\newtheorem{definition}{Definition}
\newtheorem{corollary}{Corollary}

\newtheorem{Theorem}{Theorem}
\newtheorem*{Theorem*}{Theorem}
\newtheorem{Example}{Example}

\DeclareMathOperator{\Sym}{Sym}
\DeclareMathOperator{\bSym}{\mathbb{S}ym}
\DeclareMathOperator{\Hh}{\mathcal{H}}
\DeclareMathOperator{\Rad}{\bf{Rad}}
\DeclareMathOperator{\BG}{\bf{G}}

\title{Contracting The Well-Rounded Retract}
\author{Oliver Gjoneski}
\date{}
\begin{document}
\maketitle

\begin{abstract} In this paper we present a method for contracting the well-rounded retract for $GL_2$ and $GL_3$ with a forward look to generalizing this approach in higher rank.  We also present an application of this result in computing cohomology groups with coefficients, and announce forthcoming results in this field.
\end{abstract}

\section*{Introduction}
Invariant spines have emerged as an important tool in the study of arithmetic varieties.  More specifically, in \cite{MR747876}, and more recently in \cite{MR1480546}, \cite{MR2630015}, authors have successfully demonstrated how one can use these objects to investigate the cohomology of locally symmetric spaces.

We would like to make use of these spines in order to develop a computationally efficient framework for investigating Eilenberg-MacLane group cohomology (more about this in \S \ref{Section:Applications}). As presented in \S\ref{Contraction to Filling}, an immediate obstacle in doing so is finding an algorithm to contract these spines in finite time.

\subsection{Spines}\label{Subsection:IntroSpines} In this paper we consider spaces $X = \Gamma \backslash D$, where $D$ is a non-compact globally symmetric space, and $\Gamma$ is a discrete group of automorphisms.  An active area of research is finding a deformation retract $D_0 \subset D$ of dimension equal to the virtual cohomological dimension of $D$, such that $\Gamma \backslash D_0$ is compact (see \cite{MR2305611}).  Such a subset $D_0$ is called a \emph{spine} of $D$.  

To illustrate how spines arise in a concrete setting, consider the case when $D_m$ is the space of positive definite quadratic forms on $\mathbb{R}^m.$  Elements of the arithmetic group $\gamma \in \Gamma = GL_m(\mathbb{Z})$ act on this space on the left via:
$$Q \mapsto \left(\,^t\gamma\right)^{-1}Q\gamma^{-1},$$
giving rise to a family of non-compact, locally symmetric spaces $X_m = GL_m(\mathbb{Z}) \backslash D_m.$  As early as 1907, Voronoi used the theory of perfect quadratic forms to cellulate the space $D_m$ in such a way that $GL_m(\mathbb{Z})$ acts cellularly and partitions the collection of cells in finitely many equivalence classes.  In this setting, we can introduce the well-rounded retract $W_m$ of $D_m$ as an intersection dual to the Voronoi complex (see Figure \ref{fig:WRRandVoronoi}).  It is a lower-dimensional subset of $D_m$ that is invariant under the action of $GL_m(\mathbb{Z}).$  As shown in \cite{MR747876}, the well-rounded retract is a spine of $D_m.$
\begin{figure}
\begin{center}
\includegraphics[scale=0.5,trim=5mm 12mm 0mm 80mm, clip]{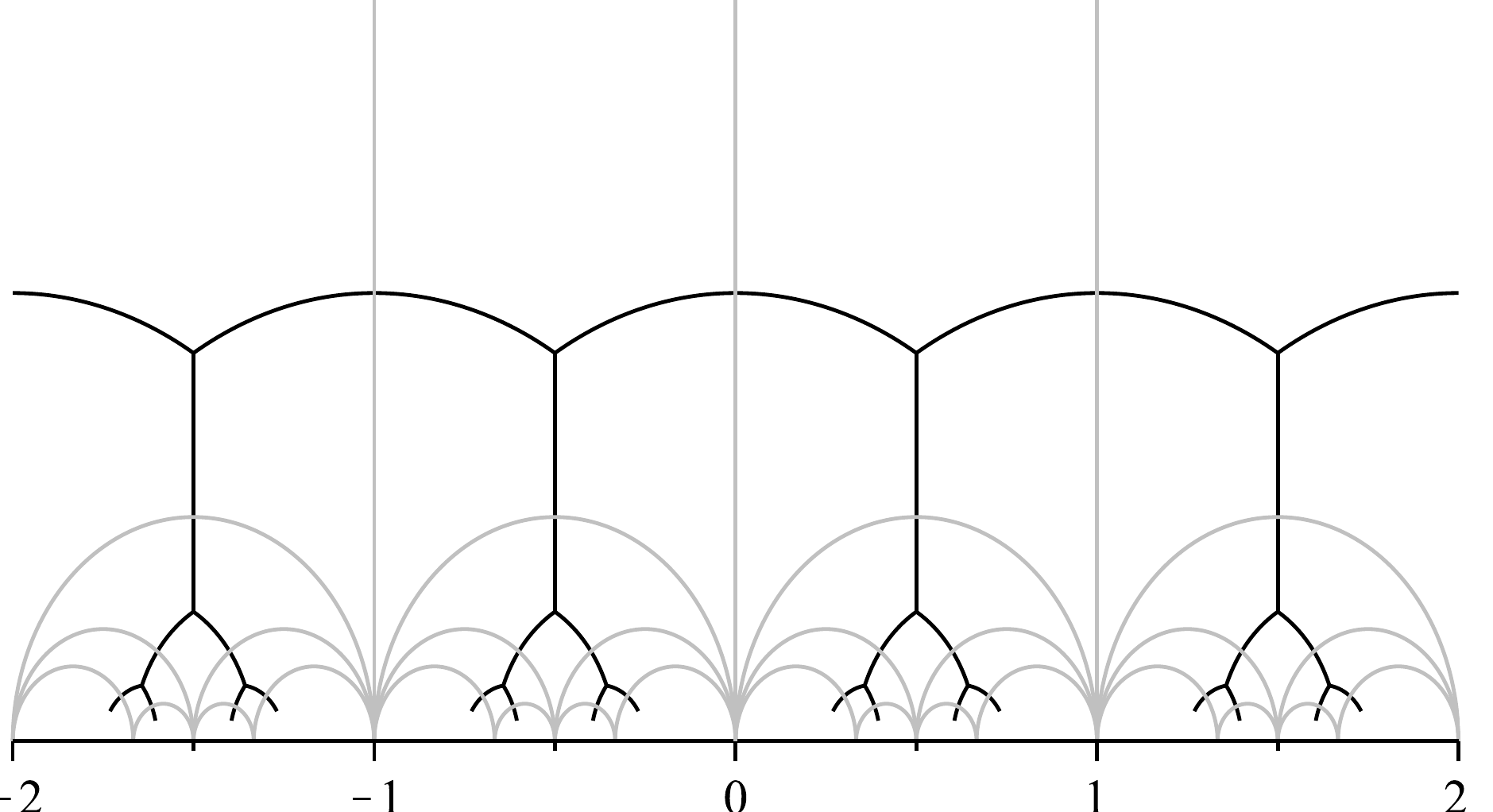} \caption{The well-rounded retract (black) and the Voronoi complex (grey) in the upper half-plane.}
\label{fig:WRRandVoronoi}
\end{center}
\end{figure}
Top-dimensional cells in the Voronoi complex, and by duality vertices in the well-rounded retract, are in one-to-one correspondence with perfect quadratic forms.  Therefore, being able to manipulate the well-rounded retract hinges on our ability to classify perfect quadratic forms on $\mathbb{R}^m$.  The number of $GL_m(\mathbb{Z})$-equivalence classes grows rapidly with $m$, with more than 10000 non-equivalent perfect quadratic forms for $m=8$ already (\cite{MR2289048}). The task of classifying all cells in the well-rounded retract, not just vertices, is an even more daunting task, with conclusive results available only up to $m = 6$ (\cite{MR1931508}).

\subsection{The problem}Contracting the well-rounded retract, a deformation retract of a contractible space, is of course possible.  What is not immediately obvious however, is how to contract $W_m$ algorithmically.  Furthermore, we aim to find a contraction that is, in a sense, invariant with respect to the action of $GL_3(\mathbb{Z})$ (for details on the type of invariance we are trying to impose, see part (2) of Theorem \ref{Theorem:MainTheorem}).   In $D_2$, where the well-rounded retract is topologically equivalent to a tri-valent tree, this problem is misleadingly easy to tackle.  In fact, when thinking of $W_2$ as the Bruhat-Tits building for $SL_2(\mathbb{Q}_2)$, there is a natural notion of shortest distance, and consequently, a contraction of the building itself along paths of shortest distance.  However, as early as $W_3 \subset D_3$ we discover that notions such as shortest distance between cells, do not generalize easily.  In fact, there is no convenient building structure to aid us in contracting this cell complex.  In this paper we present an algorithm for a very specific and combinatorial contraction of the well-rounded retract for $GL_2$ and $GL_3$, with a forward view to generalizing this algorithm to all cases where the cells of the well-rounded retract have been classified up to the action of the full arithmetic subgroup.

\subsection{Layout}In Section 1 of this paper we fix some of the notation used throughout.  In Section 2 we provide background information on the well-rounded retract, and in Section \ref{AptsAndWRR} we introduce a point of tangency between the well-rounded retract and a spherical building related to the Borel-Serre stratification of the globally symmetric space.  In Section 3, we present the details of the contraction algorithm.  We conclude the paper with Section 4, where we discuss applications of the contraction algorithm and announce forthcoming results pertaining to the cohomology of $GL_2(\mathbb{Z})$ and $GL_3(\mathbb{Z}).$

\section{Notation}
Despite the fact that in this paper we focus on $GL_m$, we present the symmetric space notation for the general case.  Namely, we use $\bf{G}$ to denote a connected, reductive, algebraic group
defined over $\mathbb{Q}$, and let $G$ be the group of real points, $\bf{G}(\mathbb{R})$.  We note that the radical of $\bf{G}$, $\Rad(\bf{G})$ is also defined over $\mathbb{Q}$, and we require $\bf{G}$ modulo $\Rad(\BG)$ to have
a strictly positive $\mathbb{Q}$-rank.  By this we mean that the dimension of the maximal $\mathbb{Q}$-split algebraic
torus of $\BG \slash \Rad(\BG)$ is strictly greater than zero.

For $\BG = GL_m$ we use $(\rho_m, V_m)$ to denote the standard representation of $GL_m(\mathbb{C})$, and $(\rho_m^n, \Sym^n(V_m))$ for the $n$-th symmetric product of the standard representation.

Let $\bf{A}_{\BG}$ denote the maximal $\mathbb{Q}$-split torus in the center $\bf{Z}_{\BG}$ of $\BG$, and $A_{\BG}$ the identity component of $\bf{A}_{\BG}(\mathbb{R})$.  We use
$K \subset G$ to denote a maximal compact subgroup, and let
$$D_{G} = G\slash K A_{\BG},$$
be the associated globally symmetric space.  For $\Gamma,$ an arithmetic subgroup of $\BG(\mathbb{Q}),$ we use $\,_{\Gamma}X_G$ to denote the corresponding locally symmetric space, namely,
$$\,_{\Gamma}X_G = \Gamma \backslash D_{G}.$$  When it is clear from context that $\BG = GL_m,$ and $\Gamma=GL_m(\mathbb{Z})$ we will abbreviate notation and use $D_m$ and $X_m$, as in \S \ref{Subsection:IntroSpines}, to denote $D_{G}$ and $\,_{\Gamma}X_G$, respectively.

Finally, when discussing the sheaf cohomology groups $H^\bullet(\,_{\Gamma}X_G, \mathbb{V}),$ we use $\mathbb{V}$ for the local system on $\,_{\Gamma}X_G$ associated to the $\Gamma$-module $(\rho,V)$.  More specifically, if we let $\mu:D_G \rightarrow \,_{\Gamma}X_G$ be the standard projection, then $\mathbb{V}$ is the locally constant sheaf defined on an open set $U \in \,_{\Gamma}X_G$ as
$$\mathbb{V}(U) = \left\{s:\mu^{-1}(U) \rightarrow V \;\middle\vert\; \begin{matrix}s \text{ is locally constant}\\s(\gamma x) = \gamma . s(x), \gamma \in \Gamma\end{matrix}\right\}.$$
In Section \ref{Section:Applications}, we will find it beneficial to realize the groups $H^\bullet\left(X_m, \bSym^n(V_m)\right)$ via differential forms.  In particular, one can show that these can be computed as the cohomology groups of the complex 
$$\Omega^\bullet\left(GL_m(\mathbb{Z})\backslash D_m, \bSym^n(V_m)\right) = \left\{\omega \in \Omega^\bullet\left(D_m, \Sym^n(V_m)\right) \;\middle\vert\; \begin{matrix}L_{\gamma}^*\omega=\rho(\gamma)\omega\\ \gamma \in GL_m(\mathbb{Z}\end{matrix}\right\}.$$
\section{The Well-Rounded Retract}
Let $G = GL_m(\mathbb{R}),$ $K = O_m(\mathbb{R}), \Gamma = GL_m(\mathbb{Z})$ and $A_G$ be the group of scalar matrices corresponding to the positive real homotheties.  Let
$$L_0 := \left\{a_1e_1 + a_2e_2 + \dots + a_ne_n \mid a_i \in \mathbb{Z}, i = 0, 1, \dots, n \right\} \cong \mathbb{Z}^m,$$
where $\{e_i\}_{i=1}^m$ is the standard basis in $\mathbb{R}^m$.

Throughout, it will be useful to think of $D_m$ as both the space of positive definite quadratic forms on $\mathbb{R}^m$ modulo homothety, and as the space of marked lattices in $\mathbb{R}^m$ modulo rotation and homothety. 

Namely, for each $g \in G$, we define the \emph{associated quadratic form} as,
$$Q_g = (\,^tg)^{-1} g^{-1}.$$
We note that $Q_g = Q_{gk}.$  Keeping in mind that a positive definite symmetric matrix can be diagonalized using orthogonal matrices yielding a $(\,^tg)^{-1} g^{-1}$ decomposition, it is apparent that the space of positive definite quadratic forms modulo homothety can be identified with $D_m.$

Similarly, for $g \in G,$ we can define the \emph{corresponding marked lattice} as,
\begin{align*}
f_g:L_0 &\rightarrow \mathbb{R}^n \\
      v &\mapsto g^{-1} v.
\end{align*}
In addition, we identify two marked lattices that differ by a homothety, namely, $\tilde{f} \sim f$ whenever $\tilde{f} = a \cdot f$, where $a \in A_G.$  Therefore we have established the following relationship,

\begin{align*}\left\{\begin{matrix}\text{Positive definite quadratic forms} \\ \text{modulo homothety}\end{matrix}\right\}\hspace{0.1 in} \cong\hspace{0.1 in} &G\slash KA_G \hspace{0.1 in}\cong\hspace{0.1 in} \left\{\begin{matrix}\text{Marked lattices} \\ \text{modulo rotation}\end{matrix}\right\}\\
Q_g= (\,^tg)^{-1} g^{-1}\mathrel{\raisebox{.2em}{\rotatebox[origin=c]{180}{$\longmapsto$}}} \hspace{0.15 in} &gKA_G \hspace{0.2 in}\longmapsto [f_g].
\end{align*}
For a positive definite quadratic form $Q_g$ fixed within its homothety class, we define the \emph{arithmetic minimum} of $Q_g$ as,
$$m_g = m(Q_g) = \text{min}\left\{\sqrt{Q_g(v)} \;\middle\vert\; v \in L_0, v \neq 0\right\}.$$   The set of \emph{minimal vectors} of $Q_g$ is defined to be 
$$M_g = M(Q_g)=\left\{v \in L_0 \;\middle\vert\; \sqrt{Q_g(v)} = m_g\right\}\slash \sim,$$
where $v_1 \sim v_2 \iff v_1 = \pm v_2.$  Hereafter, we abuse notation and write $v \in M_g$, when we mean $[v] \in M_g.$ Equivalently, following the exposition in \cite{MR1480546}, we can define these as,
$$m_g =m(f_g)= \text{min}\left\{|f_g(v)|\;\middle\vert\; v \in L_0, v \neq 0\right\},$$ $$M_g = M(f_g)=\left\{v \in L_0 \;\middle\vert\; |f_g(v)| = m_g\right\} \slash \sim ,$$
where $| \cdot |$ is the norm with respect to the standard Euclidean inner product.  

It is an easy exercise to show that for $m_g,M_g$ as above:
\begin{enumerate}
 \item The quantity $m_g$ is finite and positive for all $g \in G.$
 \item The set $M_g$ is finite and non-empty.
 \item For $v = (v_1, \dots, v_m)^t \in M_g, \gcd(v_1, \dots, v_m) = 1.$
\end{enumerate}

\begin{definition}
Integral vectors in $\mathbb{R}^m$ having coefficients that are mutually prime are called \emph{primitive}.
\end{definition}
The sum of our previous observations yields that for $g\in G,$ the vectors in $M_g$ are all primitive.

\begin{definition}
 A marked lattice $f_g$ is \emph{well-rounded} if $M_g$ spans $L_0$ as a $\mathbb{Z}$-module.
\end{definition}
\begin{definition}
The \emph{well-rounded retract} in $X_m$ is the set $W_m$ of well-rounded elements.
\end{definition}
We note that $\Gamma$ acts on $W_m$.  More specifically, let $g \in G$ be such that $f_g$ is well-rounded with $\{v_i\}_{i=1}^{m} \in M(f_g)$ spanning $L_0$ as a $\mathbb{Z}$-module.  For $\gamma \in \Gamma,$ $\{\gamma v_i\}_{i=1}^{m}$ is a set of minimal vectors for $f_{\gamma g}$ of maximal rank.  Therefore, $\gamma g$ gives rise to a well rounded marked lattice.

Finally, the importance of the well-rounded retract is summarized in the following theorem, which pertains to a map $r:[0,1] \times D_m \rightarrow D_m$ first introduced by Ash in \cite{MR747876}:
\begin{Theorem*}[\cite{MR747876}]
The map $r$ is a $GL_m(\mathbb{Z})$-invariant deformation retraction of $D_m$ onto $W_m.$
\end{Theorem*}
\subsection{Visualizing the well-rounded retract for $GL_2$ and $GL_3$}
In \cite{MR0470141} and \cite{MR747876}, Soul{\'e} and Ash present a cellular decomposition of the well-rounded retract which allows us to visualize the spine in low rank.  In particular, each cell is determined by a set of minimal integral vectors shared by all quadratic forms in the cell. Therefore, for $D_m,$ the top-dimensional cells in the well-rounded retract are decorated by $m$-integral, primitive vectors that span $L_0$ as a $\mathbb{Z}$-module.  The cells of dimension one less, are decorated by $(m+1)$ such vectors, and the list continues on.  Each cell is a closed, convex linear set in the globally symmetric space.  To see this, note that if the quadratic forms $Q_1$ and $Q_2$ share $v \in L_0$ as a minimal vector then, for $0\le\lambda\le1,$ it is certainly the case that
$$Q_\lambda = \lambda Q_1 + (1-\lambda)Q_2,$$
also has $v$ as a minimal vector.
\subsection{The well-rounded retract for $GL_2$}
Let $g = \left( \begin{matrix} a & b \\ c & d \end{matrix}\right) \in GL_2(\mathbb{R}).$  Using standard manipulation, we may factor
$$ g = \left( \begin{matrix} y^{\frac{1}{2}} & xy^{-\frac{1}{2}} \\ 0 & y^{-\frac{1}{2}} \end{matrix}\right) \left(\begin{matrix}\lambda & 0 \\ 0 & \lambda \end{matrix}\right) k,$$
where $k \in O_2(\mathbb{R})$, and $y,\lambda > 0$.  Making use of the action of $SL_2(\mathbb{R})$ on the upper half-plane by fractional linear transformations, we can identify $gKA_G$ with the point $x + iy$ in the upper half plane $\Hh^+$.

On the other hand, $g$ gives rise to a marked lattice $f_g : L_0 \rightarrow \mathbb{R}^2$, where $v \mapsto g^{-1}v.$  We note that since, $$g^{-1} = k^{-1}\left(\begin{matrix}\lambda & 0 \\ 0 & \lambda \end{matrix}\right)^{-1}\left(\begin{matrix}y^{-\frac{1}{2}} & 0 \\ 0 & y^{-\frac{1}{2}}\end{matrix}\right) \left( \begin{matrix} 1 &- x \\ 0 & y\end{matrix}\right),$$ modulo homothety and rotation, the marked lattice associated to  the coset $gKA_G$ is given by $v \mapsto \left( \begin{matrix} 1 & -x \\ 0 & y\end{matrix}\right)v.$  In order to be consistent with the coordinates on $G \slash KA_G$ introduced previously, we again identify this coset with the point $x + iy.$  We note that, we can read off all relevant information about the lattice $f_g$ in this manner.  Namely $\left(\begin{matrix}1\\0\end{matrix}\right)$ is fixed, and $\left(\begin{matrix}0\\1\end{matrix}\right)$ is mapped to $-x + iy$.  Therefore, we can use coordinates on the upper half-plane to identify $D_m$ when thought of as the space of marked lattices modulo homothety, as in \cite{MR1480546}.  Note that, the point $z = \alpha + i \beta$ with $\beta>0,$ represents the marked lattice:
$$f:L_0 \longrightarrow \mathbb{R}^2$$
$$\left(\begin{matrix}1\\0\end{matrix}\right) \mapsto \left(\begin{matrix}1\\0\end{matrix}\right),$$
$$\left(\begin{matrix}0\\1\end{matrix}\right) \mapsto \left(\begin{matrix} -\alpha \\ \beta \end{matrix}\right).$$
This identification between points in the upper half-plane and equivalence classes of marked lattices allows us to visualize the well-rounded retract.  In Figure \ref{fig:UpperHalfPlane}, each one-dimensional cell in the well-rounded retract is decorated with the minimal vectors for the marked lattices corresponding to points in that cell.

\begin{figure}[h]
\begin{center}
\includegraphics[trim=40mm 175mm 40mm 25mm, clip]{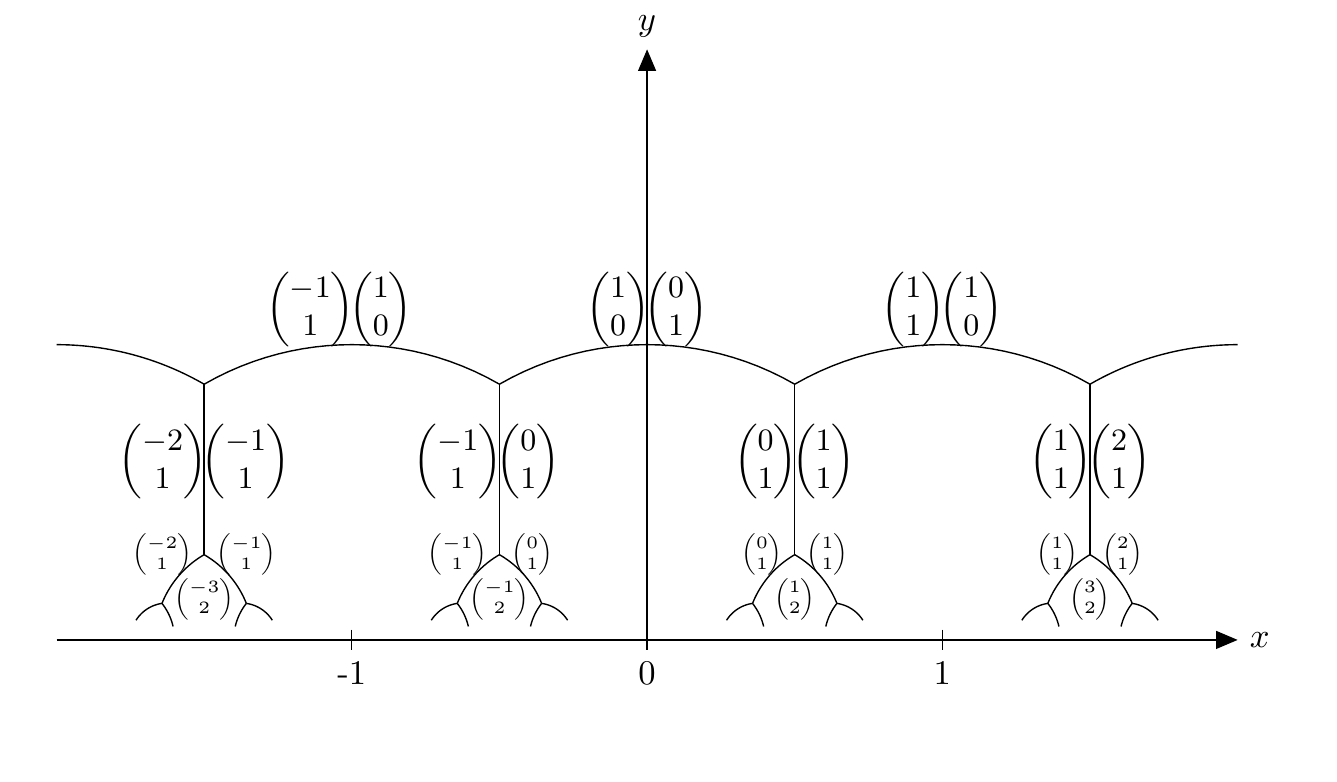} \caption{Well-rounded retract in the upper half-plane.}
\label{fig:UpperHalfPlane}
\end{center}
\end{figure}

This illustration also contains all information necessary to decorate each 0-cell in the well-rounded retract with a set of minimal vectors.  In particular, the 0-cell given as the intersection between the three $1$-cells decorated with,
$$\left\{\left(\begin{matrix}1\\0 \end{matrix}\right), \left(\begin{matrix}0\\1 \end{matrix}\right)\right\}, \left\{\left(\begin{matrix}1\\0 \end{matrix}\right), \left(\begin{matrix}1\\1 \end{matrix}\right)\right\}, \left\{\left(\begin{matrix}0\\1 \end{matrix}\right), \left(\begin{matrix}1\\1 \end{matrix}\right)\right\},$$
represents a marked lattice whose set of minimal vectors is given as the union of the three sets above,
$$\left\{\left(\begin{matrix}1\\0 \end{matrix}\right),\left(\begin{matrix}0\\1 \end{matrix}\right), \left(\begin{matrix}1\\1 \end{matrix}\right)\right\}.$$
Furthermore, once we have decorated the well-rounded retract, we can use the above illustration to visualize the action of $\Gamma$ on the individual cells, even without knowing the action of $\Gamma$ on the ambient space.  Namely, we make use of the fact that the action on marked lattices can be thought of as an action on minimal vectors, as explained previously.  Therefore, it follows that the element $\left(\begin{matrix}0 & -1 \\ 1 & 0 \end{matrix}\right) \in \Gamma,$ sends the 1-cell decorated with $\left\{\left(\begin{matrix}1\\0 \end{matrix}\right), \left(\begin{matrix}1\\1 \end{matrix}\right)\right\}$ to the one decorated with
$$\left(\begin{matrix}0 & -1 \\ 1 & 0 \end{matrix}\right)\left\{\left(\begin{matrix}1\\0 \end{matrix}\right), \left(\begin{matrix}1\\1 \end{matrix}\right)\right\} = \left\{\left(\begin{matrix}0\\1 \end{matrix}\right), \left(\begin{matrix}-1\\1 \end{matrix}\right)\right\}.$$
This is easily verified using what we know about the action of $SL_2(\mathbb{R})$ on the upper half-plane.
As a final note, we define the \emph{fundamental arc} of the well-rounded retract in the upper half-plane to be the set of points,
$$\{x+iy \mid x^2 + y^2 = 1, -1\slash2\le x \le 1\slash2\}.$$
\subsection{The well-rounded retract for $GL_3$} \label{Subsection:GL3}
Picturing the well-rounded retract in rank two, a three dimensional spine embedded in five dimensional space, is a more challenging task. Here however, we can lean on the work in \cite{MR0470141}, where by using Euclidean coordinates on the space of quadratic forms, Soul{\'e} offers an illustration of a single top-dimensional cell in the well-rounded retract for $GL_3.$  Namely, we identify the quadratic form
$$Q(u, v, w)= \left(\begin{matrix}2 & w & v \\ w & 2 & u \\ v & u & 2\end{matrix}\right),$$
with the triplet $(u, v, w).$  As described before, each cell is a convex linear set, and using the coordinates $(u, v, w)$ in Figure \ref{fig:SouleCubeMapleDecorated} we can visualize the unique three cell in the well-rounded retract containing the equivalence class of the quadratic form given by the identity matrix.

\begin{figure}
\begin{center}
\includegraphics[scale=0.7,trim=35mm 155mm 25mm 25mm, clip]{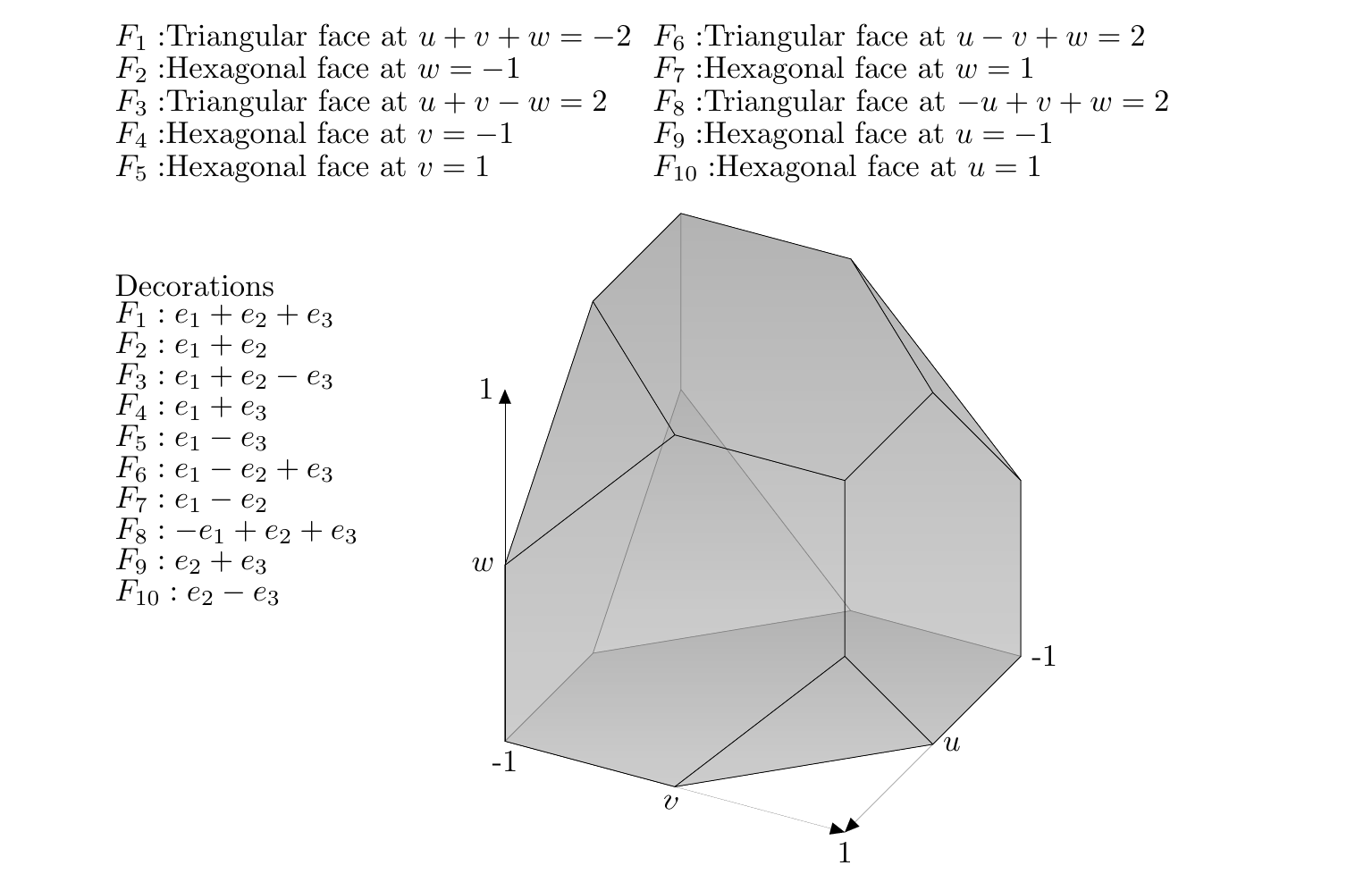} \caption{Top-dimensional cell in the Well-rounded retract for $GL_3$.}
\label{fig:SouleCubeMapleDecorated}
\end{center}
\end{figure}

Each quadratic form $Q$ in the interior of the cell in Figure \ref{fig:SouleCubeMapleDecorated} has a set of minimal vectors, $M(Q) = \{e_1, e_2, e_3\}$, where $\{e_i\}_{i=1}^{3}$ is the standard basis in $\mathbb{Z}^3.$  In the figure we have also named and decorated the 2-cells in the boundary of the cube.  In particular, the cell $F_i$ is decorated by the set $\{e_1, e_2, e_3, v_i\}$, where $F_i:v_i$ is listed to the left of the truncated cube. Therefore, for each $Q \in F_i$ we have $M(F_i):=M(Q) = \{e_1, e_2, e_3,v_i\}.$ Similarly as before, $\{e_1, e_2, e_3, v_i, v_j\}$ is the set of minimal vectors for quadratic forms in the interior of the 1-cell given as the intersection of $F_i$ and $F_j.$  We note that up to $\Gamma$-equivalence there is exactly one 3-cell, one 1-cell, and one 0-cell.  On the other hand, there are two $\Gamma$-equivalence classes of 2-cells, each visible in Figure \ref{fig:SouleCubeMapleDecorated} as a triangle, or a hexagon.

In light of Figure \ref{fig:SouleCubeMapleDecorated}, hereafter we will often refer to a generic top-dimensional cell in the well-rounded retract as a Soul{\'e} cube, a truncated cube, or just a cube.  Analogous to the rank one case, we call the cube decorated by $\{e_1, e_2, e_3\}$ the \emph{fundamental cube} in $W_3.$  Each such cube has ten 2-dimensional faces: six hexagons and four triangles.  Each hexagon is shared between three cubes, and each triangle between four cubes.  To make this more specific, consider the hexagon $F_{10}$ in Figure \ref{fig:SouleCubeMapleDecorated} determined by $u=1$.  As explained, the quadratic forms in the interior of $F_{10}$ are characterized by the set of minimal vectors, $$M(F_{10}) = \left\{e_1, e_2, e_3, \left(\begin{matrix}0 \\-1 \\ 1 \end{matrix}\right)\right\}.$$  Apart from the cube portrayed above, the two other cubes sharing this hexagonal face are decorated with the different rank three subsets of $M(F_{10})$,  namely,
$$\left\{e_1, e_2, \left(\begin{matrix}0 \\-1 \\ 1 \end{matrix}\right)\right\},\text{ and }\left\{e_1, e_3, \left(\begin{matrix}0 \\-1 \\ 1 \end{matrix}\right)\right\}.$$
Similarly, the four cubes tangent to the triangular face $F_3$ decorated with the vector $v = \left(\begin{matrix}1 \\ 1 \\-1\end{matrix}\right)$ and the vectors $e_1, e_2,$ and $e_3,$ are the cubes corresponding to the four different rank three subsets of $\{e_1, e_2, e_3, v\}.$ In summary, each cube is tangent to twenty four other cubes in codimension 1 faces.  Table \ref{Tab:SouleComplexIncidence} offers a summary of the incidence analysis, taking into account lower dimensional cells.   For $1\le i,j\le 5$ and $i> j$, $a_{ij}$ is the number of column type $j$ cells contained in the boundary of a cell of row type $i.$  On the other hand, $a_{ji}$ is the number of cells of column type $i$ containing a cell of row type $j$ in its boundary.
\begin{table}[h] 
\caption{Table of Incidences in $W_3$.}
\centering
\begin{tabular}{c|c|c|c|c|c}
\hline\hline
 & vertex &  edge & triangle & hexagon & Soul{\'e} cube \\
\hline
vertex&$-$&6&3&12&16\\
edge&2&$-$&1&4&8\\
triangle&3&3&$-$&$-$&4\\
hexagon&6&6&$-$&$-$&3\\
Soul{\'e} cube&16&24&4&6&$-$\\
\hline
\end{tabular}
\label{Tab:SouleComplexIncidence}
\end{table}

We note that the information in Table \ref{Tab:SouleComplexIncidence} departs from similar tables found in standard references, such as those in \cite{MR1463705}, and \cite[Appendix A]{MR2289048}, in the value in the top right corner.  We expand on how this value was obtained in the Appendix.

Finally, when studying the action of $\Gamma$ on the well-rounded retract, it is again very useful to think of $\Gamma$ as acting on the set of decorations.  Therefore, $\gamma \in \Gamma$ maps the cube decorated with $\{v_1, v_2, v_3\}$, to the one determined by the set $\{\gamma(v_1), \gamma(v_2), \gamma(v_3)\}.$  Note that $\gamma \in GL_3(\mathbb{Z})$ guarantees that this set has $\mathbb{R}$-rank equal to three.  This immediately points to the fact that each cube has a non-trivial stabilizer under the action of $\Gamma.$  In particular, the cube seen in Figure \ref{fig:SouleCubeMapleDecorated} is stabilized by all monomial elements in $\Gamma.$  In Figure \ref{fig:FundamentalDomainSouleCube}, we can see the fundamental domain for this action, triangulated as four neighbouring tetrahedra supported on the center of the cube.

\begin{figure}
\begin{center}
\includegraphics[scale=0.7,trim=45mm 175mm 45mm 25mm, clip]{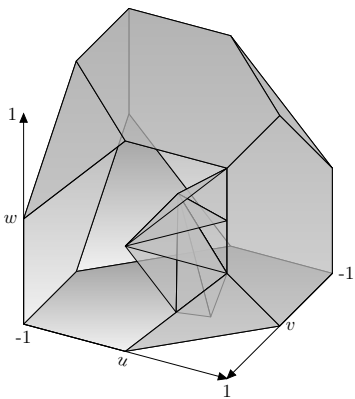} \caption{Fundamental domain for the action of $\Gamma$ within the Soul{\'e} cube.}
\label{fig:FundamentalDomainSouleCube}
\end{center}
\end{figure}

\subsection{Apartments and the well-rounded retract}\label{AptsAndWRR}
In this section we make the connection between cells in the
well-rounded retract and spherical apartments in the building associated to the boundary of the Borel-Serre stratification of $D_m.$
Recall, in the Borel-Serre stratification of $D_m$, for each
$\mathbb{Q}$-rational parabolic subgroup $P$ of $G$, one adjoins a
face $e(P)$ to the globally symmetric space.

We note that in rank one, each non-trivial rational parabolic
subgroup $P_{\{v\}}$ is the stabilizer of a flag $\{ 0 \subset
\mathbb{Q}\{v\} \subset \mathbb{Q}^2\}$, and a spherical apartment
can be represented by a pair of nodes, each node corresponding to a
(maximal) parabolic subgroup. On the other hand, as seen in Figure
\ref{fig:UpperHalfPlane}, each 1-cell in $W_2$ is decorated by a pair of
vectors $\{v_1, v_2\}.$  Therefore, to each 1-cell in the
well-rounded retract we can associate a unique spherical apartment,
namely $\{P_{\{v_1\}}, P_{\{v_2\}} \}.$  Going even further, one can show that if we were to extend the
well-rounded retraction $r$ to the Borel-Serre boundary, say via a
procedure making use of the tilling introduced in \cite{MR1470087}
as in \cite{MR1480546}, then points in $r\left(1,\left(e\left(P_{\{v_1\}}\right)\right)\right) \cap r\left(1,\left(e\left(P_{\{v_2\}}\right)\right)\right)$ describe precisely the $1$-cell in
$W_2$ decorated by $\{v_1, v_2\}.$

\begin{figure}[h!]
\begin{center}
\includegraphics[trim=50mm 185mm 50mm 25mm, clip]{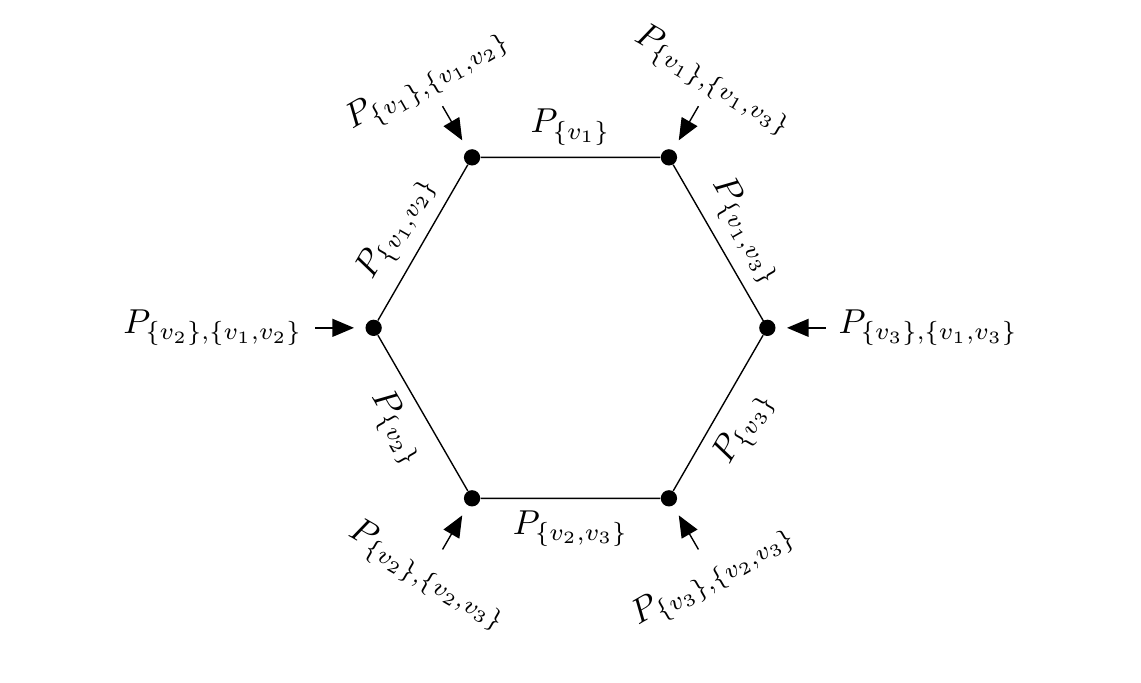} \caption{(Dual of a) rank two spherical apartment.}
\label{fig:RankTwoSpherical}
\end{center}
\vskip -0.5 in
\end{figure}

One can make a similar argument in the rank two case.  Here, there are two conjugacy classes of maximal parabolic subgroups, as well as one conjugacy class of minimal parabolics.  Recycling the notation from the previous paragraph, we use $P_U$ to denote the maximal parabolic subgroup stabilizing $U \subset \mathbb{Q}^3,$ and $P_{U_1, U_2}$ for the minimal parabolic subgroup stabilizing the flag $\{0 \subset U_1 \subset U_2 \subset \mathbb{Q}^3\}.$  For the $\mathbb{R}$-linearly independent set $\{v_1, v_2, v_3\}$, we can visualize a spherical apartment as the hexagon in Figure \ref{fig:RankTwoSpherical}, with the lines corresponding to maximal, and the points to minimal parabolic subgroups.  In fact, the spherical apartment is classically illustrated as the dimensional dual to Figure \ref{fig:RankTwoSpherical}, but for the purposes of this exposition we do not make this distinction.

As discussed before, a 3-cell $S$ in the well-rounded retract is decorated by a triple of vectors  $\{v_1, v_2,v_3\},$ making up a $\mathbb{Z}$-basis of $\mathbb{Z}^3$. To this cell we associate the spherical apartment in Figure \ref{fig:RankTwoSpherical}.  The relationship between the spherical building at infinity and the well-rounded retract goes beyond this identification.  Namely, consider a point $p \in S$ corresponding to the well-rounded marked lattice $f_{g}.$  Let $g = p_{ij}k,$ with $p_{ij} \in P_{\{v_i, v_j\}}.$  Note that $p_{ij}|_{\mathbb{R}\{v_i, v_j\}} \in \text{Aut}(\mathbb{R}\{v_i, v_j\}).$  It is not difficult to see that $p_{ij}|_{\mathbb{R}\{v_i, v_j\}}$ gives rise to a well-rounded marked lattice $f_{p_{ij}|_{\mathbb{R}\{v_i, v_h\}}}:\mathbb{Z}\{v_i, v_j\} \rightarrow \mathbb{R}\{v_i, v_j\},$ with $M(f_{p_{ij}|_{\mathbb{Q}\{v_i, v_j\}}}) = \{v_i, v_j\}\subset \mathbb{Z}\{v_i, v_j\}.$  We will refer to $p_{ij}|_{\mathbb{R}\{v_i, v_j\}}$ as the \emph{projection} of $g$ onto $\mathbb{R}\{v_i, v_j\}.$  The projection is unique up to $K \cap P_{\{v_i, v_j\}},$ an ambiguity that has no effect on the equivalence class of the marked lattice $f_{p_{ij}|_{\mathbb{Q}\{v_i, v_j\}}}$ modulo rotations.

The case when $\{v_1, v_2, v_3\} = \{e_1, e_2, e_3\}$, as seen in Figure \ref{fig:SouleCubeMapleDecorated}, is informative. The associated spherical apartment has six maximal parabolic subgroups, three of which stabilize a two-dimensional subspace: $P_{\{e_1, e_2\}},P_{\{e_1, e_3\}}$, and $P_{\{e_2, e_3\}}.$ For a point $p$ in the cube, we consider the three projections $p_{ij}|_{\mathbb{R}\{e_i, e_j\}}, 1\le i < j \le 3.$ A simple, albeit tedious, calculation shows that all three projections are in the corresponding fundamental arc for the relevant two dimensional subspace.  Furthermore, the $u$ coordinate of the point parametrizes the position of the projection to the fundamental arc in the $\mathbb{R}\{e_2, e_3\}$ subspace.  Similarly, the $v$ and $w$ coordinates parametrize the projections to the fundamental arcs in the $\mathbb{R}\{e_1, e_3\},$ and $\mathbb{R}\{e_1, e_2\}$ subspaces, respectively.

We can generalize this to a method for associating spherical apartments to top dimensional cells in the well-rounded retract in higher rank as well.

\section{Contracting the Well-Rounded-Retract}
We begin by saying a few words about why notions of distance from the well-rounded retract in the upper-half plane do not generalize conveniently to higher rank.
In particular, begin by choosing $S_0$, a top-dimensional cell to serve as an end-point for the contraction. We call a finite collection of top-dimensional cells $\{S_i\}_{i=0}^{k}$ in  $W_n$ an $r$-\emph{string} if $S_j$ and $S_{j+1}$ intersect in a co-dimension $r$ cell, and no cell is repeated in the sequence.  Furthermore we define the length of such a string to be $n$. Let $\Delta^n_r:W_n \rightarrow \mathbb{Z}$ be an integer-valued distance function defined as
$$\Delta^n_r(S) = \min\left\{\text{length}\left(\{S_i\}\right)\right\},$$
where the minimum ranges over all $r$-strings terminating with $S$.  The distance function $\Delta^2_1$ assigns what we can only call ''expected'' values to cells in the upper half-plane, and we observe that for each cell of distance $k$ in $W_2$ there is a unique string of length $k$ terminating at that cell.  The next step in constructing an explicit contraction of $W_2$ is to define subsets 
$$\Delta^2_1(k) =\{S \in W_2 \mid \Delta^2_1(S) \le k\} \subset W_2.$$
Finally, we can define a contraction to $S_0 = \Delta^2_1(0)$ recursively by specifying appropriate contractions, 
\begin{align}\label{eq:intermediate_contraction_map}
\left[\frac{1}{2^k}, \frac{1}{2^{k-1}}\right]\times \Delta^2_1(k) &\rightarrow \Delta^2_1(k)\\ \notag \left\{\frac{1}{2^{k-1}}\right\}\times \Delta^2_1(k) &\mapsto \Delta^2_1(k-1),
\end{align}
where $k \in \mathbb{Z}^{+}$.  Note, since two cells of distance $k$ in $W_2$ are incident in at most a subset of a cell of distance $k-1$, defining a map as in \eqref{eq:intermediate_contraction_map} reduces to contracting a generic $1$-cell in $W_2$ to one of its vertices. 

In $W_3$ on the other hand, there are already over 2400 cells of distance $3$ when using the most rigid of these distance functions $\Delta^3_1$, and the number quickly balloons to unmanageable heights when using $\Delta^3_r$ for $r \neq 1.$  Furthermore there is no hope for uniqueness of shortest strings.  An illustrative example is that of the cell decorated by the set
$$\left\{\left(\begin{matrix}0\\0\\1\end{matrix}\right), \left(\begin{matrix}1\\0\\-1\end{matrix}\right), \left(\begin{matrix}0\\1\\-2\end{matrix}\right)\right\}.$$
This cell in $\Delta^3_1(3) \setminus \Delta^3_1(2)$ is incident to the distance two stratum in five different 2-dimensional faces.  Complicating matters further, it is incident to other cells in $\Delta^3_1(3)$ in eight of its ten 2-dimensional faces.  Indeed, with this notion of distance, there is little hope in being able to define a contraction algorithm by simply specifying how to contract a generic top-dimensional cell to a subset of its boundary.

We tackle this challenge by re-defining distance in $W_2$ (see Definition \ref{W2:distance_def}).  We then define an analogous notion of distance in $W_3$, one that allows us to construct the contraction recursively, as in \eqref{eq:intermediate_contraction_map}.

\subsection{The Serre Tree} \label{Subsection:SerreTree}
Before we explain the contraction algorithm, we need to prove the following results.
\begin{lemma}\label{NormRankOne}
Let $g = \left( \begin{matrix} a & b \\ c & d \end{matrix}\right) = (v_1 \mid v_2) \in SL_2(\mathbb{\mathbb{Z}})$ give rise to a point on the well-rounded retract.  If $g$ neither stabilizes the fundamental arc in the well-rounded retract in $\Hh^+$, nor maps it to one of its neighboring arcs,
then,
\begin{enumerate}
 \item The sets $\{v_1, v_2, v_1+v_2\}, $ and $\{v_1, v_2, v_1-v_2\}$ are totally ordered with respect to the Euclidean norm;\label{test}
 \item $\min\{|v_1 \pm v_2|\} < \max\{|v_1|, |v_2|\}.$
\end{enumerate}
\end{lemma}

\begin{proof}
Observe that if $w_1$, and $w_2$ are the rows of $g^{-1}= \left( \begin{matrix} d & -b \\ -c & a \end{matrix}\right)$, then $|v_1|,|v_2|,$ and $|v_1 \pm v_2|$ are equal to $|w_2|, |w_1|,$ and $|w_1 \pm w_2|,$ respectively.  Consequently, it suffices to prove the lemma assuming $v_1$ and $v_2$ are the rows of $g^{-1}.$  Furthermore, since acting by $k \in SO_2(\mathbb{R})$ on the right rotates the rows, it is sufficient to prove the lemma for the rows of $\left( \begin{matrix} y & x \\ 0 & 1 \end{matrix}\right) = g^{-1}k\left(\begin{matrix}\lambda & 0 \\ 0 & \lambda\end{matrix}\right)$.  Therefore, after modding out by rotation and homothety, we can assume that the marked lattice:

\begin{align*}f_{g^{-1}}:L_0 &\longrightarrow L_0\\
 v &\mapsto g(v)
\end{align*}
maps $e_1$ to $e_1$, and $e_2$ to a point $x + iy.$ 

Let us assume by contradiction that one (or both) of the conditions in the lemma are false.  We claim that this is equivalent to saying that $g$ gives rise to a point inside the region in the upper half-plane shown in Figure \ref{fig:FundamentalRegionUpperHalfPlane}.
First assume that condition (1) is violated by having $|v_1| = |v_2|.$  This equality translates to $x^2 + y^2 = 1.$ The only part of the unit circle intersecting the well-rounded retract is the fundamental arc, meaning that $f_{g^{-1}}$ corresponds to a point in the upper half-plane that is simultaneously in the $SL_2(\mathbb{Z})$-orbit of $z=i,$ and on the fundamental arc. Clearly this point is $z=i$ itself, and $g^{-1}$ (and therefore $g$) stabilizes the fundamental arc.  Now assume that condition (1) is violated, and without loss of generality let $|v_1 + v_2|$ be equal to $|v_i|, i= 1$ or $2.$  Since, $\tilde{g} = (v_1 \pm v_2 \mid v_i)$ is again an element of $SL_2^{\pm1}(\mathbb{Z})$ the above argument applies, and as a result we can deduce that $\tilde{g}$ stabilizes the fundamental arc.  Since $g = \tilde{g}\left(\begin{matrix}1 & 0 \\ \pm 1 & 1 \end{matrix}\right)$, we conclude that $g$ maps the fundamental arc to a neighboring arc.

On the other hand, if we assume that condition (2) is violated, we note that the norm inequality $|v_1 \pm v_2|> \max\{|v_1
\, |v_2|\}$ translates to the following condition on the basis $\{f_{g^{-1}}(e_1),f_{g^{-1}}(e_2)\}$ of the lattice $f_{g^{-1}}(L_0),$
$$(x\pm 1)^2 + y^2 > |v_1| = \max\{1, x^2 + y^2\}. $$
An immediate consequence is that $|x| < \frac{1}{2}.$  In addition, it follows from $(x\pm 1)^2 + y^2 > 1$ that $x + iy$ is a point above the the circular arcs in Figure \ref{fig:FundamentalRegionUpperHalfPlane}. 

\begin{figure}[h]
\begin{center}
\includegraphics[scale=0.4, trim=2mm 0mm 2mm 30mm, clip]{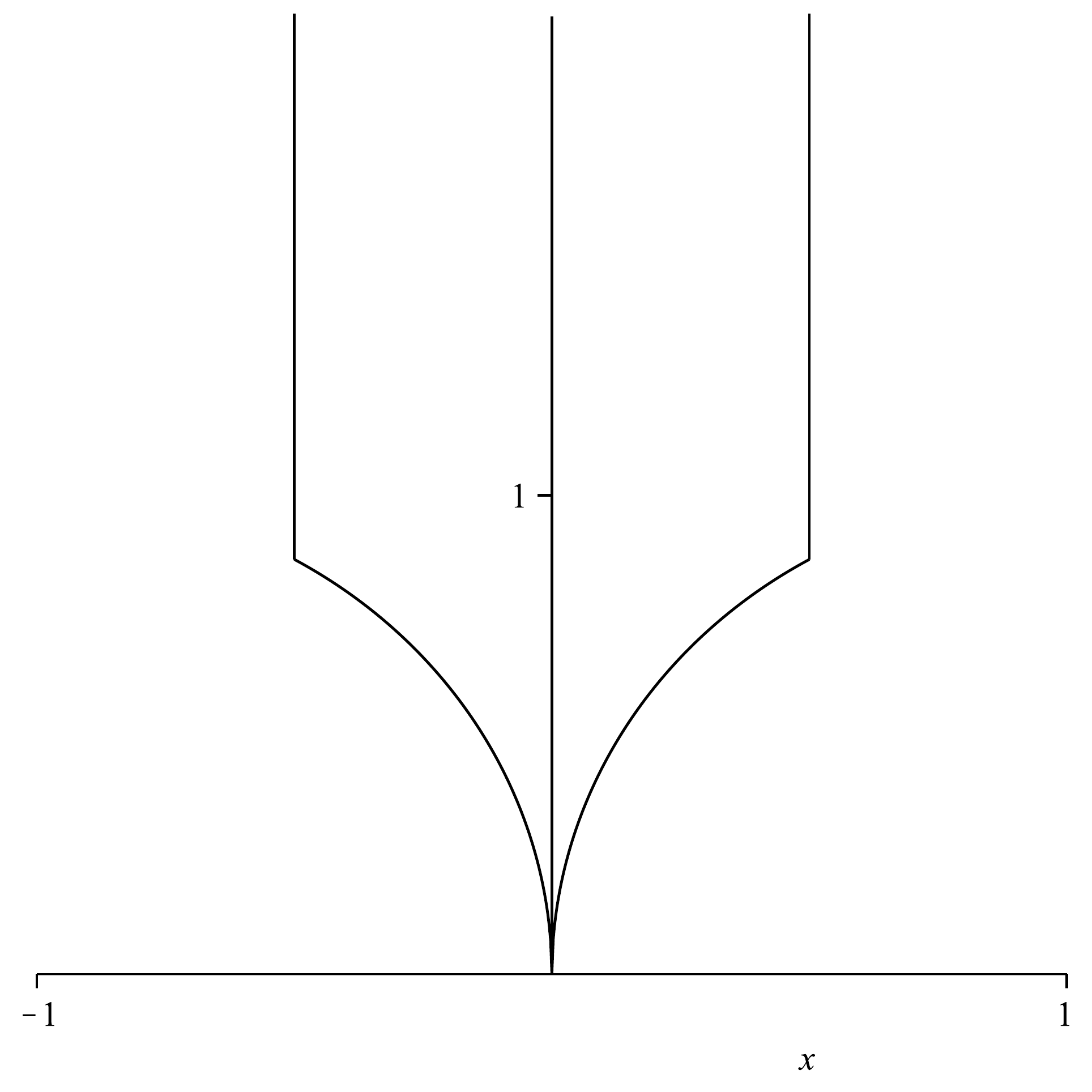}\caption{Region in the upper half-plane}
\label{fig:FundamentalRegionUpperHalfPlane}
\end{center}
\end{figure}

Therefore if any one of conditions (1) or (2) is violated, the marked lattice $f_{g^{-1}}$ corresponds to a point in the region outlined in Figure \ref{fig:FundamentalRegionUpperHalfPlane}.  Since the well-rounded retract intersects this region only in the fundamental arc, and taking into account that $g^{-1} \in SL_2(\mathbb{Z})$, it follows that $g^{-1}$ stabilizes the point $i$ when acting by fractional linear transformations, and therefore the fundamental arc.  This concludes the proof the lemma.

\end{proof}
We will also make use of the following geometric lemma,
\begin{lemma}\label{NormRankOne1}
 For $v_1, v_2 \in \mathbb{R}^m \setminus \{0\}$, if $|v_1 + v_2| \le \max\{|v_1|, |v_2|\},$ then
$$|v_1 - v_2| > \max\{|v_1|, |v_2|\}.$$
\end{lemma}

\begin{proof}
 Without loss of generality, assume $|v_1| = \max\{|v_1|, |v_2|\}.$  Then the condition $|v_1 + v_2| \le \max\{|v_1|, |v_2|\}$ is equivalent to $\frac{(v_1, v_2)}{|v_2|^2} \le -\frac{1}{2}.$  On the other hand, the conclusion $|v_1 - v_2| > \max\{|v_1|, |v_2|\}$ is equivalent to $\frac{(v_1, v_2)}{|v_2|^2} < \frac{1}{2}.$  The proof of the lemma is immediate.
\end{proof}

We impose a preorder on the set of 1-cells in $W_2.$  More specifically, when comparing two 1-cells decorated by $\{v_1, v_2\}$ and $\{w_1, w_2\}$, respectively, we say that $\{v_1, v_2\} \prec \{w_1, w_2\}$ whenever 
$$\min\{|v_1|,|v_2|\} < \min\{|w_1|,|w_2|\}, \text{  or}$$
$$\min\{|v_1|,|v_2|\} = \min\{|w_1|,|w_2|\}, \max\{|v_1|,|v_2|\} < \max\{|w_1|,|w_2|\}.$$

Consider an arbitrary 0-cell $o \in W_2,$ one that is not incident to the fundamental arc.  Since all vertices in $W_2$ are $SL_2(\mathbb{Z})$-equivalent to the one decorated by $\{e_1, e_2, e_1+e_2\},$ we can assume without loss of generality that $o$ is decorated by $\{v_1, v_2, v_1+v_2\}$ for $(v_1 \mid v_2) \in SL_2(\mathbb{Z}).$  As a consequence of Lemma \ref{NormRankOne}, we see that the three 1-cells incident at $o$ are totally ordered with respect to $\prec.$

\begin{definition}
 For an arbitrary point $p$ in the well-rounded retract, the \emph{minimal arc for p} is the smallest $1$-cell in its star.
\end{definition}

Next we consider an arbitrary, non-fundamental 1-cell $A$, decorated by $\{v_1, v_2\}$, where $(v_1 \mid v_2) \in SL_2(\mathbb{Z}).$  Consequently, the two 0-cells that make up the boundary of $A$ are decorated by $\{v_1, v_2, v_1+v_2\},$ and $\{v_1, v_2, v_1-v_2\}.$   It follows from Lemmas \ref{NormRankOne} and \ref{NormRankOne1} that exactly one of these 0-cells is incident to a 1-cell that is smaller than $A$ with respect to $\prec.$

\begin{definition}
 For a $1$-cell $A$, the \emph{minimal set for A}, $\Xi(A)$, is the set consisting of the single vertex incident to $A$ such that its minimal arc is different than $A$.  Cells in $\Xi(A)$ are called \emph{minimal cells for S.}
\end{definition}

\begin{definition} \label{W2:distance_def}
 To each point $p \in W_2$ we assign an integer $d(p)$ called the distance from $p$ to the fundamental arc, defined as 
$$d(p):=\min_A\{D(A)\},$$
where the minimum is taken over all $1$-cells $A \in W_2$ that are incident to $p$.  This definition depends on the definition of $D(A)$, similarly called the distance from $A$ to the fundamental arc, defined as
$$D(A) := \begin{cases}0 & A\text{ is the fundamental arc}, \\ \max_{p\in \Xi(A)}\{d(p)\} +1 & \text{otherwise.}\end{cases}$$
One can check that $d$ and $D$ are well-defined, and that for $p \in A \setminus \Xi(A)$, $d(p) = D(A),$ as expected.
\end{definition}
\begin{definition} The subset $W_2(n) \subset W_2, n = 0,1,2,3,\dots$ is defined as 
$$W_2(n):=\{p \in W_2 \mid d(p) \le n\}.$$
For notational convenience we define the set $W_2(-1)$ to be the point $i \in \Hh^+.$
\end{definition}
With this information in mind, we are ready to define the method of contraction. 
\begin{algorithm}\label{Algorithm:W2} Algorithm for assigning trajectories to points in $W_2$:
\begin{enumerate}[leftmargin=0.2 in]
\item An arbitrary, non-fundamental 1-cell $A$, decorated by $\{v_1,v_2\}$ can be oriented according to a direction pointing towards the fundamental arc in the well rounded retract. More specifically, we choose the direction to point towards the vertex in $\Xi(A).$ \label{rule-one} 
\item An arbitrary 0-cell should follow a trajectory in the direction specified by its minimal arc. \label{rule-two}
\item To each point $p$ in the fundamental arc, we can assign a unique trajectory terminating at $z = i.$
\end{enumerate}
\end{algorithm}

In following with steps \ref{rule-one} and \ref{rule-two} above, to each point $p\in W_2(n) \subset W_2$, we associate a unique trajectory terminating at $W_2(n-1)$.  Using Euclidean length in the upper half-plane, we can parametrize these trajectories with constant speed by, $$\phi_p: [0, 1] \rightarrow W_2, \,\,\,\, \phi_p(0) = p \in W_2(n),\,\,\,\,\, \phi_p(1) \in W_2(n-1),\,\,\,n=0,1,2,3,\dots$$

\begin{definition}\label{Definition:h2}
The contraction of the well-rounded retract in rank one is defined recursively,
\begin{align*}
 h_{2}:[0, 1] \times W_2 &\longrightarrow W_2\\
\left[\frac{1}{2^{n+1}}, \frac{1}{2^{n}}\right]\times W_2(n) &\longmapsto  W_2(n), n=0,1,2,\dots\\
(t, p) &\longmapsto \phi_p\left(2^{n+1}t-1\right).
\end{align*}
\end{definition}

\subsection{The Soul{\'e} Complex}

In this section we generalize the algorithm we employed for the Serre tree to the higher dimensional Soul{\'e} complex.  Before we introduce the specifics of the contraction, we prove a collection of results pertaining to the combinatorial nature of the decorations of $W_3$.

Hereafter, $S$ will denote an arbitrary 3-cell decorated by the triplet $\{v_1, v_2, v_3\},$ with $g = (v_1 \mid v_2 \mid v_3) \in SL_3(\mathbb{Z}).$  At times, we will use Figure \ref{fig:SouleCubeMapleDecorated} to visualize $S$.  We will abuse notation and refer to the coordinate system as $(u, v, w)$, with the caveat that the coordinates in reality are $g$-translates of the $(u, v, w)$-system used to describe the fundamental cube decorated by $\{e_1, e_2, e_3\}.$  Similarly the face $F_{i}:v_i$ as seen in Figure \ref{fig:SouleCubeMapleDecorated}, when used in reference to $S$, will denote the hexagonal face decorated by $g\{e_1, e_2, e_3, v_i\}.$

We begin by introducing a pre-ordering on vectors in $\mathbb{Z}^3,$ where the vectors are ordered first by Euclidean norm, and then lexicographically by the size of each entry.  More specifically,
$$v = \left(\begin{matrix}v_1\\v_2\\v_3\end{matrix}\right) \prec w= \left(\begin{matrix}w_1\\w_2\\w_3\end{matrix}\right),$$
whenever, $$|v| < |w|, \text{ or } $$ $$\left\{\left(|v| = |w|\right) \wedge \left(\exists\,\, 1\le i\le 3 \mid |v_i| < |w_i|, |v_j| = |w_j|\,\,\forall \,\,1\le j < i \right)\right\}.$$
We say that $v \approx w$, whenever $|v_i| = |w_i|, 1 \le i \le 3.$  Finally, we say that $v \preceq w$ whenever $v \prec w$ or $v \approx w.$
This in turn induces a pre-ordering on finite collections of vectors. In determining which collection $\{v_1, \dots, v_k\} \neq \{w_1, \dots, w_k\}$ is smaller with respect to the induced ordering $\preceq$, we choose $v_M \in \{v_1, \dots, v_k\} $ and $w_M \in \{w_1, \dots, w_k\}$ such that $ v_M \preceq v_i$ and  $w_M \preceq w_j,$ for $1\le i,j\le k.$  If $v_M \prec w_M$ we say that $\{v_1, \dots, v_k\} \prec \{w_1, \dots, w_k\},$ and similarly, if $w_M \prec v_M$ we say that $\{w_1, \dots, w_k\} \prec \{v_1, \dots, v_k\}.$  If $v \approx w$, then we proceed by comparing the sets $\{v_1, \dots, v_k\} \setminus \{v_M\}$ and $\{w_1, \dots, w_k\} \setminus \{w_M\}.$  If this process terminates without being able to conclude which collection is smaller, we say that $\{v_1, \dots, v_k\} \approx \{w_1, \dots, w_k\}.$

In rank one there is a unique $\mathbb{Z}$-basis $\{v_1, v_2\}$ of $\mathbb{Z}^2$ such that, $|v_1 \pm v_2| > \max\{|v_1|, |v_2|\},$ namely the one associated with the fundamental arc in the upper half-plane.  Analogously, in rank two we make the following definition:
\begin{definition}
 For a pair of primitive vectors $\{v_1, v_2\}$ in $\mathbb{Z}^3$, we say that $\{v_1, v_2\}$ forms a \emph{fundamental pair} whenever $v_1 \pm v_2 \succ v_1, v_2.$
\end{definition}
\begin{lemma}\label{NormRankTwo}
 If a pair $\{v_1, v_2\}$ of primitive vectors in $\mathbb{Z}^3$ is not fundamental, then
$$\max \{ |v_1 \pm v_2| \} > \max \{ |v_1|, |v_2|\} \ge \min \{ |v_1 \pm v_2| \}.$$
\end{lemma}
\begin{proof}
 This proof is a small adjustment to the one in Lemma \ref{NormRankOne1}.
\end{proof}

\begin{Theorem}\label{MainTheorem}
 Consider a $\mathbb{Z}$-basis of $\mathbb{Z}^3,$ $\{v_1, v_2, v_3\}.$  If all pairs of vectors $\{v_i, v_j\},\, 1\le i \neq j, \le 3$ are fundamental pairs, then $\{v_1, v_2, v_3\} \subset \{\pm e_1, \pm e_2,\pm e_3\}.$
\end{Theorem}

\begin{proof}

From the hypothesis, it follows that:
\begin{equation}\label{eq:NormInEq}
|v_i \pm v_j| \ge \{|v_i|, |v_j|\}.
\end{equation}
We note that an equivalent system of inequalities is
$$|\text{proj}_{v_i}(v_j)| \le \frac{1}{2}|v_i|.$$
Making use of the fact that $|\text{proj}_{v_i}(v_j)| = \frac{|(v_i, v_j)|}{|v_i|} ,$ we can restate the above conditions in a unified form,
\begin{equation}\label{eq:NormInEqRestated}
|(v_i, v_j)| \le \frac{1}{2}\min\{|v_i|^2, |v_j|^2\}.
\end{equation}

We proceed by contradiction, and assume $\{v_1, v_2, v_3\} \neq \{\pm e_1, \pm e_2, \pm e_3\}$.  Since $\{v_1, v_3, v_3\}$ is an integral basis for $L_0$, there exist integers $\{a_{ij}\}, 1\le i,j \le 3$ such that $e_i = \sum_{j}a_{ij}v_j.$  Observe that the matrix $A=\left(a_{ij}\right)$ relates one integral basis to another, and consequently has to be an element of $GL_3(\mathbb{Z}).$
Note,
\begin{align} \label{equation:new_main_proof_0}
\notag 1 &= \|e_1\|^2 = \left(\sum_{j}a_{1j}v_j,\sum_{j}a_{1j}v_j\right)\\\notag &= \sum_{j}a_{1j}^2\|v_j\|^2 + \sum_{1\le k<l\le3}2a_{1k}a_{1l}(v_k,v_l) \\\notag &\ge \sum_{j}a_{1j}^2\|v_j\|^2 - \sum_{1\le k<l\le3}|a_{1k}||a_{1l}|\min\left\{|v_k|^2, |v_l|^2\right\} \\ &\ge \sum_{1\le k<l\le3} \left[\frac{1}{2}\left(a_{1k}^2\|v_k\|^2 + a_{1l}^2\|v_l\|^2\right) - |a_{1k}||a_{1l}|\min\left\{|v_k|^2, |v_l|^2\right\} \right]. 
\end{align}
Each of the three terms on the right hand side is equal to $\frac{m}{2}, m \in \mathbb{Z}_{\ge 0}$, and zero only if 
\begin{equation} \label{equation:new_main_proof}
 \|v_k\| = \|v_l\|, |a_{1k}|=|a_{1l}|.
\end{equation}
Therefore exactly one of the terms is equal to zero, and the other two equal to one half.  Repeating the above calculation for $\|e_2\|^2$ and $\|e_3\|^2$, in each step we find a pair of vectors for which \eqref{equation:new_main_proof} holds true.  
  
If these pairs are always the same, say $\{v_p,v_q\}$ for some $1\le p < q \le 3$, then $a_{ip} = \pm |a_{iq}|$ for $1\le i\le 3.$  It follows that $2 \mid \det(A)$, which is a contradiction.  Consequently, $$\|v_1\| = \|v_2\| = \|v_3\| = \|v\|.$$  Getting back to a non-zero term on the right hand side of \eqref{equation:new_main_proof_0}, we find that:
\begin{align*}
 \|v\|^2\left(a_{1k}^2 + a_{1l}^2 - 2|a_{1k}||a_{1l}|\right) = 1.
\end{align*}
Therefore, $\|v\| = 1$.  This concludes the proof by contradiction.
\end{proof}

\begin{lemma}
At an arbitrary $0$-cell, the 6 integral vectors that make up its decoration are well ordered with respect to $\prec. $
\end{lemma}

\begin{proof}
 Since all 0-cells in the well rounded-retract are equivalent modulo $SL_3(\mathbb{Z}),$  we may assume that an arbitrary $0$-cell is decorated by the collection of vectors $\Sigma = \{v_1, v_2, v_3, v_1+v_2. v_1+v_3, v_1+v_2+v_3\},$ where $\{v_1, v_2, v_3\}$ is a $\mathbb{Z}$-basis of $\mathbb{Z}^3,$ and all 6 vectors are primitive.  It is evident that for any pair of these vectors $\{w_1, w_2\}\subset \Sigma$, there exists a third vector $w_3$ within this sextet such that $w_1 + w_3 = \pm w_2,$ or $w_1 - w_3 = \pm w_2.$ Now assume $w_1 \approx w_2,$ and let $w_3$ be as above.  From the properties of the ordering, it is evident that the coefficients of $w_3$ are either $0$, or twice the corresponding coefficient in $w_1$ in absolute value.  However, this immediately tells us that $w_3$ is not a primitive integral vector, since 2 divides all of its entries.  This is a contradiction and concludes the argument.
\end{proof}

\begin{corollary}
All cubes (apartments) incident at a point, are well ordered with respect to the ordering on collections of vectors induced by $\prec.$
\end{corollary}

\begin{proof}
 A cube incident at a 0-cell decorated by $\Sigma = \{v_1, v_2, v_3, v_1+v_2. v_1+v_3, v_1+v_2+v_3\}$ is in turn decorated by a triplet which is a subset of $\Sigma.$  The well ordering of $\Sigma$ tells us that if we compare two distinct triplets from $\Sigma,$ one will always be smaller than the other.
It is an easy exercise to show that the above is true for any point in the well-rounded retract, as the set decorating a $p$-cell containing an arbitrary point is always a subset of the set decorating a 0-cell in the complex.
\end{proof}
\begin{definition}
 For an arbitrary point $p$ in the well-rounded retract, the \emph{minimal cube for p} is the smallest $3$-cell in its star.
\end{definition}
\begin{definition}
 For a $3$-cell $S$, the \emph{minimal set for S} denoted $\Xi(S)$ is the collection of lower dimensional cells $C \in \partial(S)$ such that the minimal cube for points in $C$ is different than $S$.  Cells in $\Xi(S)$ are called \emph{minimal cells for S.}
\end{definition}
We caution the reader not to confuse $\Xi(S)$ with the decoration for $S,$ $\{v_1, v_2, v_3\},$ which is in fact the set of integral vectors on which quadratic forms in $S$ are minimal.

The next two results help us visualize the geometry of $\Xi(S)$.  In particular, we aim to show that the number of minimal faces of $S$ is less than or equal to five, and that $\Xi(S)$ forms a connected set.

\begin{lemma} \label{NumberMinimal}
For a generic top-dimensional cell $S \subset W_3$, at most five of its ten $2$-dimensional faces are in $\Xi(S).$
\end{lemma}

\begin{proof}
In justifying this claim, first let us focus on hexagonal faces.  Without loss of generality, consider the face defined by $v_1, v_2, v_3,$ and $v_1+v_2.$  The cubes incident to this face other than $S$ all contain $v_1+v_2$ in their decorating set. However, if for a moment we assume that $\{v_1, v_2\}$ does not form a fundamental pair, then according to Lemma \ref{NormRankTwo}, $\max \{ |v_1 \pm v_2| \} > \max \{ |v_1|, |v_2|\} \ge \min \{ |v_1 \pm v_2| \}.$  This in turn tells us that at most one of the two hexagonal faces defined by $\{v_1, v_2, v_3, v_1 \pm v_2 \},$ is potentially in the minimal set for $S$, whereas for the other (opposite) hexagonal face, $S$ is the minimal cube in its star.  If instead we are in the case where $\{v_1, v_2\}$ form a fundamental pair, then both of the hexagonal faces $\{v_1, v_2, v_3, v_1 \pm v_2 \}$ are not in $\Xi(S).$  Therefore, we can pair up the six hexagonal faces to determine that at most three of them will be in $\Xi(S)$.  A similar argument leads to pairing of the four triangular faces.  Namely, let the triangular face decorated by $\{v_1, v_2, v_3, v_1+v_2+v_3\}$ be in $\Xi(S)$.  Therefore, $|v_1 + v_2 + v_3| \le \max\{|v_1|, |v_2|, |v_3|\}.$  Without loss of generality let $|v_1| = \max\{|v_1|, |v_2|, |v_3|\}.$  Then, following an argument similar to the one in Lemma \ref{NormRankOne1}, it follows immediately that $|v_1 - (v_2 + v_3)| > |v_1|.$  Hence, the triangular face decorated by $\{v_1, v_2, v_3, v_1 - v_2 - v_3\}$ is not in $\Xi(S).$  In this fashion one can organize the four triangular faces of $S$ in two disjoint pairs such that if one face is in $\Xi(S)$, the face paired with it is not.
\end{proof}

\noindent We also address the question of connectedness of the minimal set for each apartment.

\begin{Theorem}
The cells in $\Xi(S)$ form a connected set.
\end{Theorem}
\begin{proof}
Since two opposite hexagons can not both be minimal as seen in the proof of Lemma \ref{NumberMinimal}, the set of minimal hexagons always form a connected set.  Therefore, we only need to concern ourselves with the triangles in any given 3-cell.  First we need the following result:

\begin{lemma} \label{ConnectedOneNew}
If $|v_i+v_j| \ge \max\{ |v_i|, |v_j| \}$ and $(v_k, v_i+v_j)\ge0$ for some $1 \le i, j,k \le 3$ distinct, then $|v_1 + v_2 + v_3| > \max \{|v_1|, |v_2|, |v_3|\}.$
\end{lemma}
\begin{proof}
\begin{align*}
|v_1 + v_2 + v_3| ^ 2 &= (v_1 + v_2 + v_3, v_1+v_2+v_3)\\ &= |v_k|^2 + 2 (v_k, v_i+v_j) + |v_i + v_j|^2 > |v_1|^2, |v_2|^2, |v_3|^2.
\end{align*}
\end{proof}
Now we pick up the proof of the theorem and assume that the triangular face decorated by $\{v_1, v_2, v_3, v_1 + v_2 + v_3\}$ is minimal in the 3-cell defined by $\{v_1, v_2, v_3\}$, i.e., $v_1 + v_2 + v_3 \prec \max \{v_1, v_2, v_3\}.$  To conclude the proof it would suffice to show that one of the three neighboring hexagons, decorated by $\{v_1, v_2, v_3, v_i+v_j\}\, 1\le i<j\le 3,$ is also minimal.
From the minimality of the triangle it follows that $|v_1 + v_2 + v_3| \le \max \{|v_1|, |v_2|, |v_3|\}.$  By the contrapositive of Lemma \ref{ConnectedOneNew}, it follows that for all distinct triples $1 \le i, j,k \le 3$, either $|v_i+v_j| < \max\{ |v_i|, |v_j| \}$ or $(v_k, v_i+v_j)<0.$  If for at least one such triple, we have $v_i+v_j \prec \max \{v_i, v_j\}$, then the proof is complete.  Therefore we may assume that for all distinct triples $1 \le i, j,k \le 3$, the following two relations hold:
\begin{enumerate}
 \item $v_i+v_j \succeq v_i, v_j$;
\item $(v_k, v_i+v_j)<0.$
\end{enumerate}
However note that from relation (1), it follows that for all pairs of vectors we have $(v_i ,v_j) \ge  -\frac{1}{2}\min\{|v_i|^2, |v_j|^2\}.$  On the other hand, since relation (2) yields $(v_k, v_i + v_j) = (v_k, v_i) + (v_k , v_j) < 0,$ it follows that for all pairs of vectors, $(v_i ,v_j) < \frac{1}{2}\{|v_i|, |v_j|\},$ and consequently $|v_i - v_j|>\max\{|v_i|, |v_j|\}.$  Thus for each pair of vectors, $v_i \pm v_j \succ |v_i|, |v_j|.$  By Theorem \ref{MainTheorem}, the only 3-cell for which this is true, is the fundamental cell with decorating set $\{e_1, e_2, e_3\},$ and in this case $\Xi(S) = \emptyset.$
\end{proof}

We offer one final Lemma that is related to the geometry of the set $\Xi(S)$ and is used in \S\ref{SubSect:BetterTrajectories}.  It can be summarized as saying that if in a given cube $S$ three hexagonal faces incident to the same triangle $T$ are in $\Xi(S)$, then $T\in \Xi(S)$ as well.

\begin{lemma} \label{ConnectedTwo}
If $|v_i+v_j| \le \max\{ |v_i|, |v_j| \}, 1 \le i < j \le 3,$ then 
$$|v_1 + v_2 + v_3| < \max \{|v_1|, |v_2|, |v_3|\}.$$ 
\end{lemma}
\begin{proof}
Without loss of generality let $|v_3| = \max \{|v_1|, |v_2|, |v_3|\}.$ Note that,
$$|v_i+v_3| \le |v_3| \implies (v_i, v_3) \le -\frac{1}{2}|v_i|^2,\,\,\, i=1,2. $$
Therefore,
$$(v_1+v_2,v_3) \le -\frac{1}{2}|v_1|^2 - \frac{1}{2}|v_2|^2.$$
Consequently,
\begin{align*}
|v_1 + v_2 + v_3| ^ 2 &= (v_1 + v_2 + v_3, v_1+v_2+v_3)\\ &= |v_3|^2 + 2(v_1+ v_2,v_3) + |v_1 + v_2|^2 \\
&\le |v_3|^2 - |v_1|^2 - |v_2|^2 + |v_1 + v_2|^2 < |v_3|^2. 
\end{align*}
This concludes the proof of the lemma.
\end{proof}

\begin{Example}
Consider the apartment $S$ defined by the columns $v_1, v_2,$ and $v_3$ of $\left(\begin{matrix}1 & 4 & 2 \\ 0 & 1 & 1 \\ 0 & 0 & 1\end{matrix}\right) \in SL_3(\mathbb{Z}).$  Here, clearly the minimal set of hexagonal faces are the three faces each decorated by the union of $\{v_1, v_2, v_3\}$ and one of the set $\{v_1 - v_2, v_1 - v_3, v_2 - v_3\}.$  Also minimal are the two triangles decorated by $\{v_1, v_2, v_3, v_1-v_2+v_3\},$ and $\{v_1, v_2, v_3, v_1+v_2-v_3\}.$  These are all of the minimal faces in $\Xi(S).$
\end{Example}

\begin{definition} \label{W3:new_distance_def}
 To each point $p \in W_3$ we assign an integer $d(p)$ called the distance from $p$ to the fundamental cell, defined as 
$$d(p):=\min_S\{D(S)\},$$
where the minimum is taken over all $3$-cells $S \in W_3$ that are incident to $p$.  This definition depends on the definition of $D(S)$, similarly called the distance from $S$ to the fundamental cell, defined as
$$D(S) := \begin{cases}0 & S\text{ is the fundamental cell}, \\ \max_{p\in \Xi(S)}\{d(p)\} +1 & \text{otherwise.}\end{cases}$$
One can check that $d$ and $D$ are well-defined, and that for $p \in S \setminus \Xi(S)$, $d(p) = D(S),$ as expected.
\end{definition}
\begin{definition} The subset $W_3(n) \subset W_3, n = 0,1,2,3,\dots$ is defined as 
$$W_3(n):=\{p \in W_3 \mid d(p) \le n\}.$$
For notational convenience, we define the set $W_3(-1)$ to be the point corresponding to the equivalence class represented by quadratic form $$Q(x_1, x_2, x_3) = x_1^2 + x_2^2 + x_3^2.$$
\end{definition}

The algorithm below outlining the contraction of the well-rounded retract in rank two should be compared against the contraction algorithm for $W_2$ as presented in \S\ref{Subsection:SerreTree}.
\begin{algorithm}\label{Algorithm:W3} Algorithm for assigning trajectories to points in $W_3$:
\begin{enumerate}[leftmargin=0.2 in]
\item Points in $S$ for which $S$ is minimal are assigned trajectories in $S$, terminating at $\Xi(S).$ \label{rank2-rule-one}
In particular, each cell in the well-rounded retract of co-dimension greater than zero is minimal for all but a single cube, namely, its minimal cube.  Consequently points in this cell are only assigned trajectories with respect to a single top-dimensional cube.  Therefore, to points in $\Xi(S)$ we associate trajectories outside of $S$.
\item Points in the fundamental cell are assigned linear trajectories to the center of the cube, $W_3(-1).$ \label{rank2-rule-two}
\end{enumerate}
\end{algorithm}
In accordance with steps \ref{rank2-rule-one}, and \ref{rank2-rule-two} above, to each point $p \in W_3(n) \subset W_3, n=0,1,2,\dots$ we associate a unique trajectory terminating at $W_3(n-1)$.  Using the Euclidean metric inherited from the five dimensional, globally symmetric space, we can parametrize these trajectories with constant speed as,
$$\phi_p: [0, 1] \rightarrow W_3, \,\,\,\, \phi_p(0) = p\in W_3(n),\,\,\,\,\, \phi_p(1) \in W_3(n-1).$$
\begin{definition}\label{Definition:h3}
The contraction of the well-rounded retract in rank two is defined recursively as,
\begin{align*}
 h_3:[0, 1] \times W_3 &\longrightarrow W_3\\
\left[\frac{1}{2^{n+1}}, \frac{1}{2^{n}}\right]\times W_3(n) &\longmapsto  W_3(n), n=0,1,2,\dots\\
(t, p) &\longmapsto \phi_p\left(2^{n+1}t -1\right).
\end{align*}
\end{definition}

There are details not addressed in the above construction.  Namely, the manner in which we contract each cube to its minimal set, as well as the continuity of the overall contraction of the well-rounded retract.  We offer a solution to the first problem in \S\ref{SubSect:BetterTrajectories}.  The argument that the contraction is continuous is presented in \S \ref{SubSect:Continuity}.

\subsubsection{Trajectories within each cube}\label{SubSect:BetterTrajectories}

In this section we describe one approach to continuously assigning trajectories to points inside a generic cube $S$, terminating at $\Xi(S).$  Here, we mean continuity as it relates to the space of paths inside one of these top-dimensional cells.  

There is more than one way to tackle this problem, and perhaps the most natural is to write down a system of ordinary differential equations modelled in such a way that the faces $\Xi(S)$ act as \emph{attractors} for neighbouring points in $S$. This can be done without too much difficulty, however the solutions to this system, which are functions parametrizing trajectories in our contraction, are not transparent nor easily manipulated.  In particular, there is no guarantee that such trajectories will satisfy any sort of invariance under the action of $GL_3(\mathbb{Z})$.  

We aim to develop a way of assigning trajectories within each cube in such a way, so that for $k = 0, 1, 2$, the union of all trajectories swept out by points in each $k$-cell equals a union of $(k+1)$-cells, each a translate of one found in the tetrahedra making up the fundamental domain in Figure \ref{fig:FundamentalDomainSouleCube}; see Theorem \ref{Theorem:MainTheorem} in \S \ref{SubSect:Specification}. 

In the non-generic case of the fundamental cube, we specify that all points are to follow trajectories to the center of the cube along straight line segments.  In the generic case, as before $S$ is a cube in the Soul{\'e} complex decorated by $\{v_1, v_2, v_3\}$.  In specifying trajectories for points in/on $S$, we will use the triangulation by translates of the tetrahedra seen in Figure \ref{fig:FundamentalDomainSouleCube}.  First we assign a trajectory to the center of the cube $o$ terminating at $\Xi(S)$.  There are several cases to consider:
\begin{enumerate}[leftmargin=0.3 in]
 \item If there are three minimal hexagons:
\begin{enumerate}[leftmargin=0.2 in]
\item If there is an element $g \in SL_3(\mathbb{Z})$ which both stabilizes $S$ and maps the three minimal hexagons to the ones decorated by $\{v_1, v_2, v_3, v_1 + v_2\}, \{v_1, v_2, v_3, v_2 + v_3\}, \{v_1, v_2, v_3, v_1 + v_3\},$ then to $o$ we assign the line segment to the center of the triangular face decorated by $\{v_1, v_2, v_3, v_1+v_2+v_3\},$ which is minimal by Lemma \ref{ConnectedTwo}.
\item If there is an element $g \in SL_3(\mathbb{Z})$ which both stabilizes $S$ and maps the three minimal hexagons to the ones decorated by $\{v_1, v_2, v_3, v_1 +\- v_2\}, \{v_1, v_2, v_3, v_2 - v_3\}, \{v_1, v_2, v_3, v_1 - v_3\},$ then to $o$ we assign the line segment to the vertex of $S$ that is at the intersection of the three minimal hexagons.
\end{enumerate}
\item If there are only two minimal hexagons, then to $o$ we assign the line segment to the vertex at the intersection of the two minimal hexagons and a triangular face in the cube.
\item If there is a single minimal hexagon, then to $o$ we assign the line segment to the center of the minimal hexagon.
\end{enumerate}
We use $\tilde{o} \in \Xi(S)$ to denote the terminal point of the trajectory originating at $o$.
Next we classify the tetrahedra in $S$ based on their relation to $\Xi(S),$ and $\tilde{o}.$
\begin{itemize}[leftmargin=0.2 in]
 \item Tier I: These are tetrahedra whose exterior face is on a 2-cell not in $\Xi(S).$
 \item Tier II: Tetrahedra in $S$ not in tier I or III.  These can also be classified as tetrahedra having a 2-dimensional intersection with a cell in $\Xi(S)$, sharing 1-cells with no more than one minimal hexagon, and not intersecting any minimal triangle containing $\tilde{o}$ in its interior.
\item Tier III: Tetrahedra in $S$ sharing a 1-cell with two minimal hexagons, or with a minimal hexagon and a (minimal) triangle containing $\tilde{o}$ in its interior.
\end{itemize}

In the first stage, we assign trajectories to tier I tetrahedra terminating in the closure of tier II tetrahedra.  In the second and most involved stage, we assign trajectories to points in tier II tetrahedra to the union of $\Xi(S)$ and the closure of tier III tetrahedra, and in the final, third stage we assign trajectories to points in tier III tetrahedra to $\Xi(S).$

\textbf{Stage I:} Consider a tier I tetrahedron described as a convex hull (in Euclidean space) of its vertex set. From the description of tier I tetrahedra, we know that at least one of these vertices is not in $\Xi(S).$  To each of these non-minimal vertices we assign a trajectory in the form of a line segment to the center of the cube.  Using the convex hull description, we can then assign trajectories to points in all tier I tetrahedra terminating to (the boundaries of) tier II and III tetrahedra.

\textbf{Stage II:} First we assign trajectories to tier II tetrahedra having a 2-dimensional face (support) on a minimal triangle (not containing $\tilde{o}$ in their interior by the definition of tier II tetrahedra).
\begin{itemize}[leftmargin=.2 in]
\item Case I: The triangle is flanked by only one minimal hexagon.  We can visualize this case in Figure \ref{fig:TriangleFlankedByHex} where, as an example, we have used the case where the hexagon at $v = 1$ is minimal.
\begin{figure}[h]
\begin{center}
\includegraphics[scale=0.7,trim=45mm 175mm 45mm 25mm, clip]{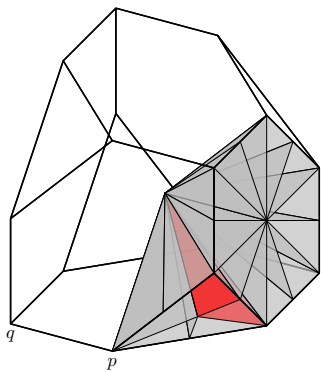} \caption{Minimal triangle flanked by a single minimal hexagon}
\label{fig:TriangleFlankedByHex}
\end{center}
\end{figure}
We pair the grey tetrahedra supported on the minimal triangle into two pairs based on the non-minimal hexagons they share in addition to the minimal triangle.  We assign trajectories to points in these two pairs of grey colored tetrahedra by projecting linearly from the centers of the two non-minimal hexagonal faces they each share a 1-cell with respectively.  The trajectories terminate at the boundaries of the red tetrahedra, the minimal triangle itself, and the 2-cell described by the points $o$, the center of the triangle, and the point $p$, found at the intersection of the triangle and the two non-minimal hexagons flanking it.  Next we assign trajectories to points in the red tetrahedra and this 2-cell by linearly projecting from the cube corner $q$ shared by the two non-minimal hexagonal faces flanking the minimal triangle.  The trajectories terminate upon entering the union of the minimal triangle and the tetrahedra having a 2-dimensional intersection with a  minimal hexagon.
\item Case II: The triangle is flanked by two minimal hexagons. We can visualize this case in Figure \ref{fig:TriangleFlankedByTwoHex} where the tetrahedra supported on the minimal hexagons are colored light gray.
\begin{figure}[h]
\begin{center}
\includegraphics[scale=0.7,trim=45mm 175mm 45mm 25mm, clip]{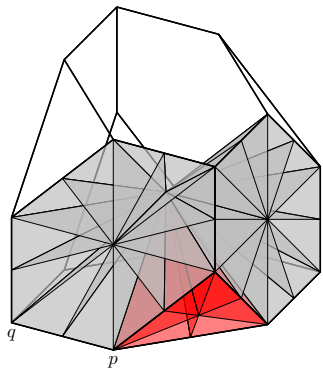} \caption{Minimal triangle flanked by two minimal hexagons}
\label{fig:TriangleFlankedByTwoHex}
\end{center}
\end{figure}
We can again use the $(u, v, w)$-coordinates from Figure \ref{fig:SouleCubeMapleDecorated}, and think of the minimal hexagons as described by $v=1$ and $u = 1,$ respectively. In this case, we assign trajectories to points in the red tetrahedra by using a linear projection from the point $(u, v, w) = (-1, -1, -1)$.  As before, these trajectories terminate at the union of the minimal triangle itself and the tetrahedra supported on the two minimal hexagons.
\end{itemize}
We fix a small $\delta>0$, and continue Stage II by focusing on the remaining tier II tetrahedra, namely those supported on a minimal hexagon.  These too come in two flavors:
\begin{itemize}[leftmargin=0.2 in]
\item Case I: Tetrahedron $T$ shares a 2-dimensional face with a tier III tetrahedron. In this case consider $C,$ the center of the 1-cell connecting $o$ to the vertex of $T$ on the minimal hexagon not on the adjacent tier III tetrahedron.  We assign trajectories to points in $T$ by linearly projecting from $\tilde{C}$, a point on the line segment connecting the center of the minimal hexagon to $C$, a distance $\delta$ outside of $T$.  The trajectories terminate at the union of the minimal hexagon, or the adjacent tier III tetrahedron.
\item Case II: Tetrahedron $T$ does not share a 2-dimensional face with a tier III tetrahedron.  This case is more involved and to gain better understanding we label the vertex set $V_1,\dots, V_4,$ with $V_1$ being the center of the associated minimal hexagon, and $V_4 = o.$  We refer the reader to Figure \ref{fig:TetrahedronContraction}.  Let $C_1$ and $C_2$ be the centers of the segments from $V_4$ to $V_2$ and $V_3$, respectively. As before let $\tilde{C}_1$ and $\tilde{C}_2$ be the points on the lines $\overline{V_1C_1}$ and $\overline{V_1C_2}$, a distance $\delta$ ``behind'' $C_1$ and $C_2$ respectively.  In this case we assign trajectories based on a linear projection from a line segment $L$, just outside of the tetrahedron, connecting $\tilde{C}_1$ and $\tilde{C}_2.$  Specifically, for a point $M$ inside the tetrahedron, we express it as a linear combination of the vectors $v_1 = V_1-V_4, v_2 = V_2 - V_4$, and $v_3 = V_3 - V_4$: $M = \sum\lambda_iv_i,$ with $\sum \lambda_i \le 1.$  Set $\tilde{M} = 0v_1 + \lambda_2v_2 + \lambda_3v_3,$ with $\lambda_2 + \lambda_3 = c<1.$  Note also, that $L$ can be described as $L = \theta_1\tilde{v}_1 + \theta_2\tilde{v}_2$, where $\tilde{v}_1 = \tilde{C}_1 - V_4, \tilde{v}_2 = \tilde{C}_2 - V_4,$ and $\theta_1 + \theta_2 = 1.$ We now pick, $p_M \in L$, based on the proportion of $\lambda_2$ to $\lambda_3.$  More specifically, we let $p_M = \frac{\lambda_2}{c}\tilde{v}_1 + \frac{\lambda_3}{c}\tilde{v}_2.$  The trajectory associated to $M$ is described by the line segment from $p_M$ to $M$ terminating at the face with vertices $\{V_1,V_2,V_3\}$ (which is in $\Xi(S)$), or the edge $\overline{V_1V_4}$ (which is in a tier III tetrahedron).

\begin{figure}[h]
\begin{center}
\includegraphics[scale=1.5,trim=75mm 225mm 75mm 25mm, clip]{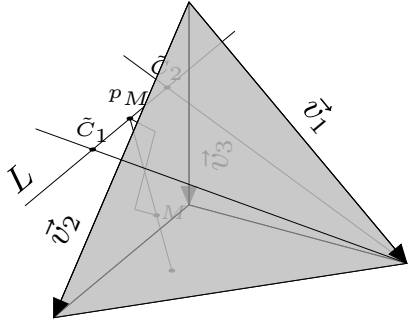} \caption{Contraction of a model tetrahedron.}
\label{fig:TetrahedronContraction}
\end{center}
\end{figure}

\end{itemize}

\textbf{Stage III:} To all points in tier III tetrahedra we assign trajectories terminating at $\Xi(S)$ by mapping $o$ to $\tilde{o}$ linearly and thinking of the tetrahedra as convex hulls of their vertex sets.

This concludes the description of trajectories terminating at $\Xi(S),$ assigned to points in $S \setminus \Xi(S)$.
\subsubsection{Continuity}\label{SubSect:Continuity}
In this section we say a few words about the continuity of $h_3$ described in Algorithm \ref{Algorithm:W3}.

Observe that since $h_3 : [0,1] \times W_3 \rightarrow W_3$ is defined recursively, with $W_3(n)$ contracting onto $W_3(n-1)$ in the time referenced by $t \in \left[\frac{1}{2^{n+1}},\frac{1}{2^{n}}\right],$ it suffices to show continuity of each of these individual mappings.

On the other hand, each intermediate map $\left[\frac{1}{2^{n+1}},\frac{1}{2^{n}}\right] \times W_3(n) \rightarrow W_3(n)$ is entirely determined by the trajectories parametrized in constant speed within each cube $S \in W_3(n)$.  There is no interaction across boundary faces: $S$ is minimal for all faces $F \in \partial(S) \setminus \Xi(S)$, and consequently other top-dimensional cells incident to $F$ are in $W_3(N) \setminus W_3(n),$ for some $N > n.$

Finally the contraction within each cube is continuous since it can be decomposed into fixed-source projections, and mappings of convex hulls defined by linear contractions on the defining vertex set.

\subsection{Properties of the Contraction}\label{SubSect:Specification}
Having defined the contraction $h_3:[0,1] \times W_3 \rightarrow W_3$, in the next theorem we explore some of its properties including how it interacts with the action of the full arithmetic subgroup $GL_3(\mathbb{Z})$.

\begin{Theorem}\label{Theorem:MainTheorem}
Let $h_3$ be as in Definition \ref{Definition:h3}.  Then,
\begin{enumerate}[leftmargin=0.3 in]
  \item $h_3$ is a ``local lift'' of the lower rank contractions defined on the Serre trees inside the Levi components of a subset of the Bore-Serre faces at infinity;
  \item $h_3 \left( [0,1], \cdot\right)$ preserves $\mathcal{C},$ where $\mathcal{C}$ is the set of 0-, 1-, 2-, and 3-cells in the union of all $GL_3(\mathbb{Z})$-translates of the four fundamental tetrahedra illustrated in Figure \ref{fig:FundamentalDomainSouleCube}.\label{Theorem:MainTheorem:InvariancePart}
\end{enumerate}
\end{Theorem}

\begin{proof}
Part (i) of the theorem follows from the fact that Algorithm \ref{Algorithm:W3} is a generalization of Algorithm \ref{Algorithm:W2}.  In particular consider a top dimensional cell in the well-rounded retract decorated by $\{v_1, v_2, v_3\}$.  Consider the two opposing hexagons decorated by $\{v_1, v_2, v_3, v_2 - v_3\},$ and $\{v_1, v_2, v_3, v_2+v_3\}$ respectively.  As in \S \ref{AptsAndWRR},  quadratic forms in $S$ remain well-rounded when, under an appropriate transformation, they are considered as quadratic forms on the 2-dimensional real subspace stabilized by elements in $P = P_{\{v_2, v_3\}}$.  Furthermore, the two hexagons in question project to two, distinct 0-cells incident to a single 1-cell in this lower-dimensional well-rounded retract.  The manner in which we decide which one of these hexagons belongs to the minimal set for this top-dimensional cell reduces to comparing the Euclidean norm of $v_2 \pm v_3$ to that of $v_2$ and $v_3.$  Compared to the contraction algorithm for $W_2 \subset D_2$, we see that this assignment to the minimal set is consistent with the manner in which we choose which vertex is minimal for the 1-cell in question, resulting in a preferred direction for the lower-dimensional contraction.  The same holds for the two other maximal parabolic subgroups conjugate to $P_{\{v_2, v_3\}}$ in the spherical apartment determined by the set $\{v_1, v_2, v_3\}.$  

Finally for part (\ref{Theorem:MainTheorem:InvariancePart}) of the theorem, let $\mathcal{C}$ be the set of $0$-, $1$-, $2$-, and $3$-cells in the union of all $GL_3(\mathbb{Z})$-translates of the four fundamental tetrahedra.  Tracing through the manner in which we contract points in $W_3$, we see that for a $0$-cell $p \in W_3,$ $h_3\left([0,1],p\right)$ is a union of $1$-cells in $\mathcal{C}.$  Similarly, if we let $l \subset W_3$ denote a $1$-cell in $\mathcal{C},$ $\cup_{p \in l}h_3\left([0,1],p\right)$ is a union of $1$- and $2$-cells in $\mathcal{C}.$  Finally for a $2$-cell $s \in \mathcal{C},$ $\cup_{p \in s}h_3\left([0,1],p\right)$ is a union of $1$-, $2$-, and $3$-cells in $\mathcal{C}.$  This concludes the proof of part (\ref{Theorem:MainTheorem:InvariancePart}).
\end{proof}

\section{Applications to Cohomology} \label{Section:Applications}
We begin by recalling the notion of Eilenberg-MacLane group cohomology for a general group $G$ with coefficients in a $G$-module $(\rho, V).$  In particular, we consider the complex $C^\bullet\left(G, V\right)$ of all functions $f:G^\bullet \rightarrow V.$  To this complex we associate the co-boundary operator $d^\bullet:C^\bullet\left(G, V\right) \rightarrow C^{\bullet+1}\left(G, V\right)$ defined as:
\begin{align*}
(d^if)(g_1,\dots,g_{i+1}) &= \rho(g_1)f(g_2, \dots, g_{i+1})\\ &+ \sum_{k=1}^{k=n}(-1)^kf(\dots, g_{k-1},g_{k}g_{k+1}, g_{k+2}, \dots) \\&+ (-1)^{n+1}f(g_1, \dots, g_n).
\end{align*}
We state without proof that this is a co-chain map, namely $(d^{i+1}\circ d^i) f = 0$.  Therefore, in line with standard notation, we say 
$$Z^i\left(G, V\right) = \left\{f \in \text{Ker}\left(d^i\right)\right\} \subset C^i\left(G, V\right)$$ are the $i$-co-chains.  Again, following standard conventions, we define co-boundaries as,
$$B^i\left(G, V\right) = \left\{f \in \text{Im}\left(d^{i-1}\right)\right\} \subseteq C^i\left(G, V\right), i\ge1,$$ and $B^0(G, V)=0.$  With this notation in place, we can compute Eilenberg-MacLane group cohomology as,
$$H^\bullet \left(G, V\right) = Z^\bullet\left(G, V\right) \slash B^\bullet\left(G, V\right).$$

\subsection{Constructing Eilenberg-MacLane co-cycles}
Consider the map
\begin{align}\label{DupontIso}
\Omega^\bullet\left(GL_m(\mathbb{Z})\backslash D_m, \mathbb{V}\right) &\rightarrow C^\bullet\left(GL_m(\mathbb{Z}),V\right),\\
\notag \omega &\rightarrow \phi_\omega
\end{align}
where $\phi_\omega: GL_m(\mathbb{Z})^k \rightarrow V$ is defined as
$$\phi_\omega(\gamma_1, \dots, \gamma_k) = \int_{\sigma(\gamma_1, \dots,\gamma_k)}\omega,$$
for $\sigma$ a \emph{filling} of the symmetric space $D_m.$ 

\begin{definition}{\cite[Ch.9]{MR0500997}}
A \emph{filling} of $D_m$ is a family of $C^\infty$ singular simplices \\
$\sigma(\gamma_1, \dots, \gamma_k) \colon \Delta^k \rightarrow D_m$, $\gamma_1,\dots,\gamma_k \in GL_m(\mathbb{Z})$, $k=0,1,2,\dots,$
such that for $k= 1, 2, \dots,$
$$\sigma(\gamma_1, \dots, \gamma_k) \circ \xi^i = \begin{cases}\begin{array}{lll} \gamma_1 \sigma(\gamma_2, \dots, \gamma_k) & &i = 0,\\
\sigma(\gamma_1, \dots, \gamma_i\gamma_{i+1}, \dots,\gamma_k) & & 0<i<k,\\\sigma(\gamma_1, \dots, \gamma_{k-1}) & & i = k. \end{array}\end{cases} $$
\end{definition}
Above, we used $\xi^i$ to denote the standard face operators associated with the regular $k$-simplex in $\mathbb{R}^{k+1}.$  Exploiting the properties of the filling outlined in the definition above, one can show as in \cite{MR0500997} that the map \eqref{DupontIso} descends to an isomorphism on cohomology:
$$H^\bullet\left(GL_m(\mathbb{Z})\backslash D_m, \mathbb{V}\right) \stackrel{\sim}{\rightarrow} H^\bullet\left(GL_m(\mathbb{Z}), V\right).$$

A key ingredient in the above recipe for constructing Eilenberg-MacLane co-cycles is the choice of a filling of $D_m$.
\subsection{As it relates to a contraction of the well-rounded retract}\label{Contraction to Filling}

In this section we use the combinatorial structure of the well-rounded retract in order to define a filling of $D_m$.  

More specifically, let $h_m:I\times W_m \rightarrow W_m$ generalize the contraction of $W_3$ defined in Defintion \ref{Definition:h3}, with $h_m(1, W_m)=o.$  We define $\sigma(\gamma_1, \dots, \gamma_p)$ inductively by
\begin{align*}
\sigma(\gamma_1, &\dots, \gamma_p)(t_0, \dots, t_p) =\\ &=\begin{cases}\begin{array}{ll}o & p=0\\h_m\left(t_1,\gamma_1\cdot o\right) & p=1\\h_m\left(1-t_0,\gamma_1\cdot \sigma(\gamma_2, \dots, \gamma_p)\left(\frac{t_1}{1-t_0}, \dots, \frac{t_p}{1-t_0}\right) \right) & p>1. \end{array}\end{cases}
\end{align*}
It follows from \cite[Ch.9]{MR0500997} that when defined as above, $\sigma$ is a filling of $D_m.$

\begin{Theorem}
Let $f \in Z^i\left(GL_m(\mathbb{Z}),V\right)$ be a closed co-cycle for $GL_m(\mathbb{Z})$ taking values in the $GL_m$-module $V.$  Then there exists a closed co-cycle $\tilde{f} \sim f$, such that all values of $\tilde{f}$ can be computed as integrals over compact cells in the well-rounded retract $W_m.$
\end{Theorem}

The importance in this result is also computational in nature.  More specifically, we could have arrived at the above result simply by quoting the fact that the well-rounded retract is a deformation retract of a contractible space and as such is contractible itself.  However, by developing an explicit contraction of $W_m$ we have built an environment in which we can concretely  describe the cells we wish to integrate in order to compute the values of the group co-cycles.  We can take this one step further, and lean on part (2) of Theorem \ref{Theorem:MainTheorem} to present the following result in the case of $m=3.$

\begin{corollary}
 Let $m=3$ and $\phi_\omega$ be as in \eqref{DupontIso}.  All values of $\phi_\omega$ can be recovered from four vectors $v_1^{\omega}, v_2^{\omega}, v_3^{\omega}, v_4^{\omega} \in V.$
\end{corollary}

An in-depth look at the ramifications of this result is forthcoming in \cite{OliverModularSymbols}, where we explore the case of $V=\bSym^n(V_m)$ for $m=2,$ and $3.$  In these special cases, our approach should be compared to similar constructions in the theory of modular symbols such as those in \cite{MR0120372} and \cite{MR2289048}.  Our methodology completes this framework by also treating symbols associated to non-cuspidal differential forms.

\section*{Acknowledgements}
I am extremely grateful to my advisor Les Saper for his infinite patience in explaining the intricacies of symmetric spaces.  I would like to thank Richard Hain who first brought period polynomials to my attention.  I would also like to thank Steven Zucker for reading an early draft of this paper and giving me some very helpful feedback.

\section*{Appendix}
The value in the top right corner of Table \ref{Tab:SouleComplexIncidence} in \S \ref{Subsection:GL3} refers to the number of top-dimensional cells (cubes) incident to a given vertex of the well-rounded retract.  Since all 0-cells in the well-rounded retract are $SL_3(\mathbb{Z})$-equivalent, it suffices to write down sixteen different cubes that contain the vertex at the intersection of the faces $F_2, F_4,$ and $F_{10}$, as seen in Figure \ref{fig:SouleCubeMapleDecorated}.  This particular vertex is decorated by the set $\{e_1, e_2, e_3, e_1 + e_3, e_1 + e_2, e_2 - e_3\}.$  We provide a list of decorations enumerating all cubes containing the vertex in question:
$$\{ e_1 , e_2 , e_3 \},\{ e_1 , e_2 , -e_2+e_3 \},\{ e_1 , e_2 , e_1+e_3 \}, \{ e_1 , e_3 , -e_2+e_3 \},$$ $$\{ e_1 , e_3 , -e_1-e_3 \}, \{ e_2 , e_3 , e_1+e_2 \}, \{ e_2 , e_3 , e_1+e_3 \},\{ e_1 , e_2 - e_3 , e_1+e_3 \},$$ $$ \{ e_1 , e_2 - e_3 , e_1+e_2 \}, \{ e_1 , e_1 + e_3 , -e_1-e_2 \}, \{ e_2 , e_2 - e_3 , -e_1-e_2 \},$$ $$\{ e_2 , e_2 - e_3 , -e_1-e_3 \},\{ e_2 , e_1 + e_2 , -e_1-e_3 \}, \{ e_3 , e_2 - e_3 , -e_1-e_3 \},$$ $$ \{ e_3 , e_2 - e_3 , -e_1-e_2 \}, \{ e_3 , e_1 + e_2 , -e_1-e_3 \}.$$
\bibliographystyle{alpha}
\bibliography{References}

\begin{thebibliography}{EVGS02}

\bibitem[AGM10]{MR2630015}
Avner Ash, Paul~E. Gunnells, and Mark McConnell.
\newblock Cohomology of congruence subgroups of {${\rm SL}_4(\Bbb Z)$}. {III}.
\newblock {\em Math. Comp.}, 79(271):1811--1831, 2010.

\bibitem[AM97]{MR1480546}
Avner Ash and Mark McConnell.
\newblock Cohomology at infinity and the well-rounded retract for general
  linear groups.
\newblock {\em Duke Math. J.}, 90(3):549--576, 1997.

\bibitem[Ash84]{MR747876}
Avner Ash.
\newblock Small-dimensional classifying spaces for arithmetic subgroups of
  general linear groups.
\newblock {\em Duke Math. J.}, 51(2):459--468, 1984.

\bibitem[Dup78]{MR0500997}
Johan~L. Dupont.
\newblock {\em Curvature and characteristic classes}.
\newblock Lecture Notes in Mathematics, Vol. 640. Springer-Verlag, Berlin,
  1978.

\bibitem[EVGS02]{MR1931508}
Philippe Elbaz-Vincent, Herbert Gangl, and Christophe Soul{\'e}.
\newblock Quelques calculs de la cohomologie de {${\rm GL}_N(\Bbb Z)$} et de la
  {$K$}-th\'eorie de {$\Bbb Z$}.
\newblock {\em C. R. Math. Acad. Sci. Paris}, 335(4):321--324, 2002.

\bibitem[Gjo]{OliverModularSymbols}
Oliver Gjoneski.
\newblock Multi-variable period polynomials.
\newblock {\em in preparation}.

\bibitem[MM89]{MR1463705}
Robert MacPherson and Mark McConnell.
\newblock Classical projective geometry and modular varieties.
\newblock In {\em Algebraic analysis, geometry, and number theory ({B}altimore,
  {MD}, 1988)}, pages 237--290. Johns Hopkins Univ. Press, Baltimore, MD, 1989.

\bibitem[Sap97]{MR1470087}
Leslie Saper.
\newblock Tilings and finite energy retractions of locally symmetric spaces.
\newblock {\em Comment. Math. Helv.}, 72(2):167--202, 1997.

\bibitem[Shi59]{MR0120372}
Goro Shimura.
\newblock Sur les int\'egrales attach\'ees aux formes automorphes.
\newblock {\em J. Math. Soc. Japan}, 11:291--311, 1959.

\bibitem[Sou78]{MR0470141}
Christophe Soul{\'e}.
\newblock The cohomology of {${\rm SL}_{3}({\bf Z})$}.
\newblock {\em Topology}, 17(1):1--22, 1978.

\bibitem[Ste07]{MR2289048}
William Stein.
\newblock {\em Modular forms, a computational approach}, volume~79 of {\em
  Graduate Studies in Mathematics}.
\newblock American Mathematical Society, Providence, RI, 2007.
\newblock With an appendix by Paul E. Gunnells.

\bibitem[Yas06]{MR2305611}
Dan Yasaki.
\newblock On the existence of spines for {$\Bbb Q$}-rank 1 groups.
\newblock {\em Selecta Math. (N.S.)}, 12(3-4):541--564, 2006.

\end{thebibliography}

\end{document}